\documentclass[11pt, reqno]{amsart} \textwidth=14.9cm \oddsidemargin=1cm \evensidemargin=1cm
\usepackage{amsmath,amssymb,amscd,mathrsfs,stmaryrd,wasysym}
\usepackage{amsmath}
\usepackage{amsxtra}
\usepackage{amscd}
\usepackage{amsthm}
\usepackage{amsfonts}
\usepackage{amssymb}
\usepackage{eucal}
\usepackage{color}
\usepackage{graphicx}%
\newtheorem{theorem}{Theorem}[section]
\newtheorem{conjecture}{Conjecture}[section]
\newtheorem{lemma}[theorem]{Lemma}

\newtheorem{corollary}[theorem]{Corollary}
\theoremstyle{definition}
\newtheorem{definition}[theorem]{Definition}
\newtheorem{remark}[theorem]{Remark}

\definecolor{A}{rgb}{.75,1,.75}

\numberwithin{equation}{section}

\begin{document}

\title[Presentation of Hecke-Hopf algebras for non-simply-laced type]{On the presentation of Hecke-Hopf algebras for non-simply-laced type}
\author[Weideng Cui]{Weideng Cui}
\address{School of Mathematics, Shandong University, Jinan, Shandong 250100, P.R. China.}
\email{cwdeng@amss.ac.cn}

\begin{abstract}
Hecke-Hopf algebras were defined by A. Berenstein and D. Kazhdan. We give an explicit presentation of an Hecke-Hopf algebra when the parameter $m_{ij},$ associated to any two distinct vertices $i$ and $j$ in the presentation of a Coxeter group, equals $4,$ $5$ or $6$. As an application, we give a proof of a conjecture of Berenstein and Kazhdan when the Coxeter group is crystallographic and non-simply-laced. As another application, we show that another conjecture of Berenstein and Kazhdan holds when $m_{ij},$ associated to any two distinct vertices $i$ and $j,$ equals $4$ and that the conjecture does not hold when some $m_{ij}$ equals $6$ by giving a counterexample to it.

In the second part of this paper, we define a new family of algebras, called generalized affine Hecke algebras, associated to affine Weyl groups such that the Hecke-Hopf algebras associated to finite Weyl groups and the Hecke algebras associated to affine Weyl groups are both subalgebras of the new kind of algebras. Finally, we study the representation theory of the generalized affine Hecke algebras.
\end{abstract}



\maketitle
\medskip
\section{Introduction}
Associated to each Coxeter group $W,$ we have an Iwahori-Hecke algebra $H_{q}(W),$ which has been extensively studied. Recently, Berenstein and Kazhdan [BK] defined a new family of Hopf algebras, called Hecke-Hopf algebras, such that $H_{q}(W)$ are left coideal subalgebras of them. Let us now recall their definitions.

Recall that a Coxeter group $W$ is generated by $s_i$ ($i\in I$) subject to relations $(s_is_j)^{m_{ij}}=1$,
where $m_{ij}=m_{ji}\in \mathbb{Z}_{\ge 0}$  are such that $m_{ij}=1$ if and only if $i=j$.

Given a Coxeter group $W=\langle s_i\:|\:i\in I\rangle,$ following [BK, Definition 1.13] we first define $\hat{\bf{H}}(W)$ to be the $\mathbb{Z}$-algebra generated by $s_i,D_i$ ($i\in I$) subject to the relations:

(i) Rank $1$ relations: $s_i^2=1$,  $D_i^2=D_i$, $s_iD_i+D_is_i=s_i-1$ for $i\in I$.

(ii) Coxeter relations: $(s_is_j)^{m_{ij}}=1.$

(iii) The linear braid relations: $\underbrace{D_is_js_i\cdots s_{j'}}_{m_{ij}} =\underbrace{s_j\cdots s_{i'}s_{j'}D_{i'}}_{m_{ij}}$
for all distinct $i,j\in I$ with $m_{ij}\ne 0$, where $i'=\begin{cases}
i & \text{if $m_{ij}$ is even}\\
j & \text{if $m_{ij}$ is odd}\\
\end{cases}$ and  $\{i',j'\}=\{i,j\}$.

By [BK, Theorem 1.15], $\hat{\bf{H}}(W)$ is a Hopf algebra with the coproduct $\Delta,$ the counit $\varepsilon,$ and the antipode $S.$

Set $\mathcal{S}:=\{ws_iw^{-1}\,|\,w\in W, i\in I\}$. Then $\mathcal{S}$ the set of all reflections in $W$. It is easy to see that linear braid relations in $\hat {\bf H}(W)$ imply that for any $s\in {\mathcal S},$ there is a unique element $D_s\in \hat {\bf H}(W)$ such that $D_{s_i}=D_i$ for $i\in I$ and $D_{s_iss_i}=s_iD_ss_i$ for any $i\in I$ and $s\in {\mathcal S}\setminus \{s_i\}$.

Let $\hat {\bf D}(W)$ be the subalgebra of $\hat {\bf H}(W)$ generated by all $D_s$ ($s\in {\mathcal S}$) and
${\bf K}(W):=\bigcap\limits_{w\in W} w\hat {\bf D}(W)w^{-1}$.

The algebra $\hat {\bf D}(W)$ can be viewed as naturally filtered by $\deg D_s=1$ for $s\in {\mathcal S}.$ For distinct $i,j\in I,$ we denote by ${\bf K}_{ij}(W)$ the set of all elements in ${\bf K}(W_{\{i,j\}})\cap Ker~\varepsilon\subset \hat {\bf D}(W_{\{i,j\}})$ having degree at most $m_{ij},$ where $W_{\{i,j\}}$ is the parabolic subgroup of $W$ generated by $s_i$ and $s_j.$

By [BK, Theorem 1.23] we have that the ideal $\underline {\bf J}(W),$ generated by  by all ${\bf K}_{ij}(W)$ ($i\ne j\in I$), is  a Hopf ideal, and thus the quotient algebra ${\bf H}(W)=\hat {\bf H}(W)/\underline {\bf J}(W)$ is a Hopf algebra.

${\bf H}(W)$ is defined to the Hecke-Hopf algebra of $W.$

The following theorem follows from [BK, Theorem 1.24].

\begin{theorem}
\label{introduction-theorem}
Suppose that $W$ is simply-laced, that is, $m_{ij}\in \{0,2,3\}$ for all distinct $i,j\in I$. Then the Hecke-Hopf algebra ${\bf H}(W)$ is generated by $s_i,D_i$ $(i\in I)$ subject to the relations:

$\bullet$
$s_i^2=1$,
$s_iD_i+D_is_i=s_i-1$, $D_i^2=D_i$ for $i\in I$

$\bullet$
$s_js_i=s_is_j$,
$D_js_i=s_iD_j$, $D_jD_i=D_iD_j$ if $m_{ij}=2$.

$\bullet$
$s_js_is_j=s_is_js_i$,
$D_is_js_i=s_js_iD_j$, $D_js_iD_j=s_iD_jD_i+D_iD_js_i+s_iD_js_i$ if $m_{ij}=3$.

\end{theorem}

Recall that an $I\times I$-matrix $A=(a_{ij})$ is a generalized Cartan matrix if $a_{ii}=2$, $a_{ij}\in \mathbb{Z}_{\le 0}$ for $i\ne j$ and $a_{ij}\cdot a_{ji}=0$ implies that $a_{ij}=a_{ji}=0$.

The following conjecture is stated in [BK, Conjecture 1.38].

\begin{conjecture}
\label{introduction-conjecture1}
Let $A=(a_{ij})_{i,j\in I}$ be a generalized Cartan matrix. Let $W=W_A$ be the corresponding crystallographic Coxeter group, i.e., $m_{ij}=
\begin{cases}
2+a_{ij}a_{ji} & \text{if $a_{ij}a_{ji}\le 2$}\\
6   & \text{if $a_{ij}a_{ji}=3$}\\
0   & \text{if $a_{ij}a_{ji}>3$}\\
\end{cases}$
for $i\ne j\in I$ and let ${\mathcal L}_I=\mathbb{Z}[t_i^{\pm 1}]_{i\in I}$. Then the assignments
\begin{equation}
\label{eq:demazure W}
s_i(t_j):=t_i^{-a_{ij}}t_j,\quad D_i(t_j):=t_j\frac{1-t_i^{-a_{ij}}}{1-t_i}\,
\end{equation}
for $i,j\in I$ turn ${\mathcal L}_I$ into an ${\bf H}(W)$-module algebra.
\end{conjecture}

Conjecture \ref{introduction-conjecture1} has been proved by Berenstein and Kazhdan in [BK, Section 7.8] when $W$ is simply-laced, that is, when $A$ is symmetric. The first result of the paper is the proof of this conjecture when $W$ is non-simply-laced.

We denote by ${\bf D}(W)$ the image of $\hat {\bf D}(W)$ under the natural projection $\hat {\bf H}(W)\twoheadrightarrow {\bf H}(W)$. The following conjecture is stated in [BK, Conjecture 1.33] which gives a presentation of ${\bf D}(W)$.
\begin{conjecture}
\label{introduction-conjecture2}
For any crystallographic Coxeter group $W$,  the algebra ${\bf D}(W)$ is generated by $D_s$ $(s\in {\mathcal S})$ subject to relations $D_s^2=D_s$ $(s\in {\mathcal S})$, the braid relations$:$
\begin{equation}
\label{braid-relations}
\underbrace{D_iD_jD_{i}\cdots}_{m_{ij}} =\underbrace{D_jD_iD_{j}\cdots}_{m_{ij}}
\end{equation}
and  the following relations$:$

$(a)$ Quadratic-linear relations $($for $1\le p<\frac{m}{2n})$$:$
$$\displaystyle{\sum\limits_{0\le a<b< \frac{m}{n}:b-a=\frac{m}{n}-p} D_{r+an} D_{r+bn} = \sum\limits_{0\le a'<b'< \frac{m}{n}:b'-a'=p} D_{r+b'n} D_{r+a'n}-
\sum\limits_{p\le c<\frac{m}{n}-p} D_{r+cn}}.$$

$(b)$ Yang-Baxter type relations $($for $0\le t\le \frac{m}{n})$$:$
$$\overset{\longrightarrow}{\prod\limits_{t\le a\le \frac{m}{n}-1}}(1-D_{r+an})
\overset{\longrightarrow}{\prod\limits_{0\le b\le t-1}} D_{r+bn}
=\overset{\longleftarrow}{\prod\limits_{0\le b\le t-1}} D_{r+bn}\overset{\longleftarrow}{\prod\limits_{t\le a\le \frac{m}{n}-1}}(1-D_{r+an}).$$
where $m :=m_{ij}\geq 2$ for any two distinct $i, j\in I,$ $n$ is a divisor of $m$, $1\leq r\leq n$ and $D_k:=D_{w\cdot \underbrace{s_is_j\cdots}_{2k-1}\cdot w^{-1}}\in {\bf D}(W)$ with any $w\in W$ such that $\ell(ws_i)=\ell(w)+1,\ell(ws_j)=\ell(w)+1$ and for $k=1,\ldots,m$.
\end{conjecture}

Conjecture \ref{introduction-conjecture2} has been proved by Berenstein and Kazhdan when $W$ is simply-laced, that is, $m_{ij}\in \{0,2,3\}$ for all distinct $i,j\in I$. Another result of this paper is the proof of the conjecture when $m_{ij},$ associated to any two distinct vertices $i$ and $j,$ equals $4$ and gives a counterexample to the conjecture when some parameter $m_{ij}$ equals $6.$

This paper is organized as follows. In Sections 2, 3 and 4, we first consider the parabolic subgroups $W_{\{i,j\}}$ of $W$ generated by $s_i$ and $s_j$ for any two distinct vertices $i$ and $j$ when $m_{ij}=4,$ $5$ or $6,$ which are isomorphic to the Coxeter groups of type $B_2,$ $G_2$ or $I_{2}(5),$ respectively. We prove that Conjecture \ref{introduction-conjecture1} is true when $W_{\{i,j\}}$ is of type $B_2$ and $G_2$ on the one hand, and on the other hand, we show that Conjecture \ref{introduction-conjecture2} holds when $W_{\{i,j\}}$ is of type $B_2$ and that it does not hold when $W_{\{i,j\}}$ is of type $G_2.$

In Section 5, by using the results in Sections 2-4, we give an explicit presentation of an Hecke-Hopf algebra when the parameter $m_{ij},$ associated to any two distinct vertices $i$ and $j$ in the presentation of a Coxeter group, equals $4,$ $5$ or $6$ and give a proof of Conjecture \ref{introduction-conjecture1} when the Coxeter group is crystallographic and non-simply-laced. Besides, we show that Conjecture \ref{introduction-conjecture2} holds when $m_{ij},$ associated to any two distinct vertices $i$ and $j,$ equals $4$ and that it does not hold when some $m_{ij}$ equals $6.$

In Section 6, we define a new family of algebras, called generalized affine Hecke algebras, associated to affine Weyl groups such that the Hecke-Hopf algebras associated to finite Weyl groups and the Hecke algebras associated to affine Weyl groups are both subalgebras of the new kind of algebras. Finally, we study the representation theory of the generalized affine Hecke algebras.

\section{Type $B_2$}
Let $W_1$ be the Coxeter group of type $B_2$ with generators $s_1,s_2$ and relations $s_1^{2}=s_2^{2}=1$ and $s_1s_2s_1s_2=s_2s_1s_2s_1.$ We then have the following theorem.

\begin{theorem}\label{theoremb1}
The Hecke-Hopf algebra $\mathbf{H}(W_1),$ associated to $W_1,$ has the following presentation$:$ it is generated by the elements $s_1,s_2,D_1,D_2$ with the following relations$:$
\begin{align}
s_1^{2}=&s_2^{2}=1,\quad s_1s_2s_1s_2=s_2s_1s_2s_1;\label{B-relation1}\\[0.1em]
D_{i}^{2}=&D_{i};\label{B-relation2}\\[0.1em]
s_{i}D_i&+D_{i}s_i=s_i-1;\label{B-relation3}\\[0.1em]
D_1s_2&s_1s_2=s_2s_1s_2D_1,\quad D_2s_1s_2s_1=s_1s_2s_1D_2;\label{B-relation4}\\[0.1em]
D_2s_1&s_2D_1=D_1D_2s_1s_2+s_1D_2D_1s_2+s_1s_2D_1D_2+s_1s_2D_1s_2+s_1D_2s_1s_2;\label{B-relation5}\\[0.1em]
D_1D_2&D_1D_2=D_2D_1D_2D_1.\label{B-relation6}
\end{align}
\end{theorem}
\begin{proof}
Our proof follows from that of [BK, Theorem 1.24], but more complicated. We just give a sketch of the main procedure.

Recall from [BK, Lemma 5.20] that for each Coxeter group $W$, $\hat {\bf D}(W)$ is a $\mathbb{Z}W$-module algebra via the following action:
\begin{align}\label{b2-module-algebra}
w(D_s)=
\begin{cases}
D_{wsw^{-1}} & \text{if } \ell(ws)>\ell(w)
\\
1-D_{wsw^{-1}}
 & \text{if } \ell(ws)<\ell(w)
\end{cases}
\end{align}
for $w\in W$ and $s\in {\mathcal S},$ and $\ell$ is the length function on $W.$

We set $\widehat{D}_1 :=D_{s_{1}}=D_1$, $\widehat{D}_2 :=D_{s_{1}s_2s_1}=s_1D_{2}s_1,$ $\widehat{D}_3 :=D_{s_{2}s_1s_2}=s_{2}D_1s_{2}$ and $\widehat{D}_{4} :=D_{s_2}=D_{2}.$ We also set $\bar{s}_{1} :=s_1,$ $\bar{s}_2 :=s_1s_2s_1,$ $\bar{s}_{3} :=s_2s_1s_2$ and $\bar{s}_{4} :=s_2.$

By [BK, Proposition 7.33(b)] we have that
\begin{align}\label{b2-k12w1}
K_{12}(W_1) \subseteq \{x\in \hat {\bf D}(W_1)_{\leq 4}\:|\:d_{k}(x)=0=\varepsilon(x), 1\leq k\leq 4\},
\end{align}
where $d_{k}(\widehat{D}_{k'})=\delta_{k, k'}$ and $d_{k}(xy)=d_{k}(x)y+\bar{s}_{k}(x)d_{k}(y)$ for all $\hat {\bf D}(W_1).$

We have that $\hat {\bf D}(W_1)$ is a free $\mathbb{Z}W$-module with a basis consisting of the following elements: $1, \widehat{D}_1, \widehat{D}_2, \widehat{D}_3, \widehat{D}_4,$ $\widehat{D}_{a}\widehat{D}_{b}$ ($a\neq b$), $\widehat{D}_{a}\widehat{D}_{b}\widehat{D}_{c}$ ($a, b, c$ are different from each other), $\widehat{D}_{a}\widehat{D}_{b}\widehat{D}_{a}$ ($a\neq b$), $\widehat{D}_{a}\widehat{D}_{b}\widehat{D}_{c}\widehat{D}_{d}$ ($a, b, c, d$ are different from each other), $\widehat{D}_{a}\widehat{D}_{b}\widehat{D}_{c}\widehat{D}_{b}$ ($a, b, c$ are different from each other), $\widehat{D}_{a}\widehat{D}_{b}\widehat{D}_{c}\widehat{D}_{a}$ ($a, b, c$ are different from each other), $\widehat{D}_{a}\widehat{D}_{b}\widehat{D}_{a}\widehat{D}_{d}$ ($a, b, d$ are different from each other), $\widehat{D}_{a}\widehat{D}_{b}\widehat{D}_{a}\widehat{D}_{b}$ ($a\neq b$).

By \eqref{b2-module-algebra} we have

$\bar{s}_{1}(\widehat{D}_1)=1-\widehat{D}_1,\quad\bar{s}_{1}(\widehat{D}_2)=\widehat{D}_4,\quad\bar{s}_{1}(\widehat{D}_3)=\widehat{D}_3,
\quad\bar{s}_{1}(\widehat{D}_4)=\widehat{D}_2;$

$\bar{s}_{2}(\widehat{D}_1)=1-\widehat{D}_3,\quad\bar{s}_{2}(\widehat{D}_2)=1-\widehat{D}_2,\quad\bar{s}_{2}(\widehat{D}_3)=1-\widehat{D}_1,
\quad\bar{s}_{2}(\widehat{D}_4)=\widehat{D}_4;$

$\bar{s}_{3}(\widehat{D}_1)=\widehat{D}_1,\quad\bar{s}_{3}(\widehat{D}_2)=1-\widehat{D}_4,\quad\bar{s}_{3}(\widehat{D}_3)=1-\widehat{D}_3,
\quad\bar{s}_{3}(\widehat{D}_4)=1-\widehat{D}_2;$

$\bar{s}_{4}(\widehat{D}_1)=\widehat{D}_3,\quad\bar{s}_{4}(\widehat{D}_2)=\widehat{D}_2,\quad\bar{s}_{4}(\widehat{D}_3)=\widehat{D}_1,
\quad\bar{s}_{4}(\widehat{D}_4)=1-\widehat{D}_4.$

\noindent By using the actions, we first claim that if an element $x\in \hat {\bf D}(W_1)_{\leq 4},$ when is written as a $\mathbb{Z}$-linear combination of the basis above, does not contain the terms $\widehat{D}_{1}\widehat{D}_{2}\widehat{D}_{3}\widehat{D}_{4},$ $\widehat{D}_{4}\widehat{D}_{3}\widehat{D}_{2}\widehat{D}_{1},$ then it suffices to consider the situation when $x\in \hat {\bf D}(W_1)_{\leq 3}.$

Assume that
\begin{align*}
x=a_{0}+\sum\limits_{l=1}^{4}a_{l}\widehat{D}_{l}+\sum\limits_{b\neq c}f_{bc}\widehat{D}_{b}\widehat{D}_{c}+\sum\limits_{b\neq c,d; c\neq d}g_{bcd}\widehat{D}_{b}\widehat{D}_{c}\widehat{D}_{d}+\sum\limits_{b\neq c}h_{bc}\widehat{D}_{b}\widehat{D}_{c}\widehat{D}_{b}.
\end{align*}
By $d_1(x)=\cdots=d_{4}(x)=0$ we get that the coefficients satisfy the following relations:

$(1)$ $a_{1}=0;\quad f_{12}+f_{41}+g_{412}=0;\quad f_{13}+f_{31}=0;\quad f_{14}+f_{21}+g_{214}=0;$

$(2)$ $g_{123}+g_{413}+g_{431}=0;\quad g_{132}+g_{312}+g_{341}=0;\quad g_{124}+g_{421}=0;\quad g_{142}+g_{241}=0;$

$(3)$ $g_{134}+g_{314}+g_{321}=0;\quad g_{143}+g_{213}+g_{231}=0;\quad g_{213}+g_{312}=0;\quad g_{231}+g_{132}=0;$

$(4)$ $a_{2}+f_{12}+f_{32}+g_{132}+g_{312}=0;\quad f_{21}-f_{32}-g_{132}-g_{312}=0;\quad f_{23}-f_{12}-g_{132}-g_{312}=0;$

$(5)$ $f_{24}+f_{42}+g_{124}+g_{324}+g_{142}+g_{342}+g_{412}+g_{432}=0;\quad g_{214}-g_{324}-g_{342}=0;$

$(6)$ $g_{241}+g_{421}-g_{432}=0;\quad g_{234}-g_{124}-g_{142}=0;\quad g_{243}+g_{423}-g_{412}=0;$

$(7)$ $a_{3}+f_{23}+f_{43}+g_{243}+g_{423}=0;\quad f_{31}+f_{13}+g_{231}+g_{431}+g_{123}+g_{143}+g_{213}+g_{413}=0;$

$(8)$ $f_{32}-f_{43}-g_{243}-g_{423}=0;\quad f_{34}-f_{23}-g_{243}-g_{423}=0;\quad g_{312}+g_{132}-g_{143}=0;$

$(9)$ $g_{321}-g_{413}-g_{431}=0;\quad g_{314}+g_{134}-g_{123}=0;\quad g_{341}-g_{231}-g_{213}=0;$

$(10)$ $a_{4}=0;\quad f_{41}+f_{34}+g_{341}=0;\quad g_{324}+g_{423}=0;\quad g_{342}+g_{243}=0;\quad f_{42}+f_{24}=0;$

$(11)$ $f_{43}+f_{14}+g_{143}=0;\quad g_{412}+g_{342}+g_{324}=0;\quad g_{421}+g_{241}+g_{234}=0;$

$(12)$ $g_{413}+g_{314}=0;\quad g_{431}+g_{134}=0;\quad g_{423}+g_{243}+g_{214}=0;\quad g_{432}+g_{142}+g_{124}=0;$

$(13)$ $h_{12}=h_{13}=h_{14}=0;\quad h_{21}=h_{23}=h_{24}=0;\quad h_{31}=h_{32}=h_{34}=0;$

$\qquad h_{41}=h_{42}=h_{43}=0.$

\noindent By solving the system consisting of the equations in $(1)$-$(12)$, we get that

$(1)$ $a_{2}=-f_{12}-f_{32}-g_{132}-g_{312};\quad f_{23}=f_{12}+g_{132}+g_{312};\quad f_{21}=f_{32}+g_{132}+g_{312};$

$(2)$ $g_{231}=-g_{132};\quad g_{213}=-g_{312};\quad g_{321}=-g_{123};\quad g_{143}=g_{132}+g_{312};\quad g_{431}=-g_{134};$

$(3)$ $g_{314}=-g_{134}+g_{123};\quad g_{413}=g_{134}-g_{123};\quad a_{3}=-f_{12}-f_{32}-g_{132}-g_{312};$

$(4)$ $g_{341}=-g_{132}-g_{312};\quad f_{43}=f_{32}-g_{243}-g_{423};\quad g_{342}=-g_{243};\quad g_{324}=-g_{423};$

$(5)$ $g_{412}=g_{243}+g_{423};\quad f_{41}=-f_{12}-g_{243}-g_{423};\quad g_{214}=-g_{243}-g_{423};\quad g_{421}=-g_{124};$

$(6)$ $f_{14}=-f_{32}-g_{132}-g_{312}+g_{243}+g_{423};\quad g_{432}=-g_{234};\quad g_{241}=g_{124}-g_{234};$

$(7)$ $g_{142}=-g_{124}+g_{234};\quad f_{34}=f_{12}+g_{132}+g_{312}+g_{243}+g_{423};\quad f_{31}=-f_{13};\quad f_{42}=-f_{24},$

\noindent where $f_{12},$ $f_{32},$ $g_{132},$ $g_{312},$ $g_{243},$ $g_{423},$ $g_{123},$ $g_{134},$ $g_{124},$ $g_{234},$ $f_{13}$ and $f_{24}$ are independent variables.

By $(1)$-$(7)$ and \eqref{b2-k12w1} we get that $K_{12}(W_1)$ is contained in the ideal $J$ which is generated by the following elements:

$(1)$ $-\widehat{D}_{2}+\widehat{D}_{2}\widehat{D}_{3}-\widehat{D}_{3}-\widehat{D}_{4}\widehat{D}_{1}+
\widehat{D}_{3}\widehat{D}_{4}+\widehat{D}_{1}\widehat{D}_{2};$

$(2)$ $-\widehat{D}_{2}+\widehat{D}_{2}\widehat{D}_{1}-\widehat{D}_{3}-\widehat{D}_{1}\widehat{D}_{4}+
\widehat{D}_{4}\widehat{D}_{3}+\widehat{D}_{3}\widehat{D}_{2};$

$(3)$ $-\widehat{D}_{2}-\widehat{D}_{3}+\widehat{D}_{2}\widehat{D}_{3}+\widehat{D}_{2}\widehat{D}_{1}+
\widehat{D}_{1}\widehat{D}_{3}\widehat{D}_{2}-\widehat{D}_{2}\widehat{D}_{3}\widehat{D}_{1}+
\widehat{D}_{1}\widehat{D}_{4}\widehat{D}_{3}-\widehat{D}_{3}\widehat{D}_{4}\widehat{D}_{1}-\widehat{D}_{1}\widehat{D}_{4}+\widehat{D}_{3}\widehat{D}_{4};$

$(4)$ $-\widehat{D}_{2}-\widehat{D}_{3}+\widehat{D}_{2}\widehat{D}_{3}+\widehat{D}_{2}\widehat{D}_{1}+
\widehat{D}_{3}\widehat{D}_{1}\widehat{D}_{2}-\widehat{D}_{2}\widehat{D}_{1}\widehat{D}_{3}+
\widehat{D}_{1}\widehat{D}_{4}\widehat{D}_{3}-\widehat{D}_{3}\widehat{D}_{4}\widehat{D}_{1}-\widehat{D}_{1}\widehat{D}_{4}+\widehat{D}_{3}\widehat{D}_{4};$

$(5)$ $\widehat{D}_{2}\widehat{D}_{4}\widehat{D}_{3}-\widehat{D}_{4}\widehat{D}_{3}-\widehat{D}_{3}\widehat{D}_{4}\widehat{D}_{2}+
\widehat{D}_{4}\widehat{D}_{1}\widehat{D}_{2}-\widehat{D}_{4}\widehat{D}_{1}-\widehat{D}_{2}\widehat{D}_{1}\widehat{D}_{4}+
\widehat{D}_{1}\widehat{D}_{4}+\widehat{D}_{3}\widehat{D}_{4};$

$(6)$ $\widehat{D}_{4}\widehat{D}_{2}\widehat{D}_{3}-\widehat{D}_{4}\widehat{D}_{3}-\widehat{D}_{3}\widehat{D}_{2}\widehat{D}_{4}+
\widehat{D}_{4}\widehat{D}_{1}\widehat{D}_{2}-\widehat{D}_{4}\widehat{D}_{1}-\widehat{D}_{2}\widehat{D}_{1}\widehat{D}_{4}+
\widehat{D}_{1}\widehat{D}_{4}+\widehat{D}_{3}\widehat{D}_{4};$

$(7)$ $\widehat{D}_{1}\widehat{D}_{2}\widehat{D}_{3}-\widehat{D}_{3}\widehat{D}_{2}\widehat{D}_{1}+\widehat{D}_{3}\widehat{D}_{1}\widehat{D}_{4}-
\widehat{D}_{4}\widehat{D}_{1}\widehat{D}_{3};$

$(8)$ $\widehat{D}_{1}\widehat{D}_{3}\widehat{D}_{4}-\widehat{D}_{3}\widehat{D}_{1}\widehat{D}_{4}-\widehat{D}_{4}\widehat{D}_{3}\widehat{D}_{1}+
\widehat{D}_{4}\widehat{D}_{1}\widehat{D}_{3};$

$(9)$ $\widehat{D}_{1}\widehat{D}_{2}\widehat{D}_{4}-\widehat{D}_{4}\widehat{D}_{2}\widehat{D}_{1}+\widehat{D}_{2}\widehat{D}_{4}\widehat{D}_{1}-
\widehat{D}_{1}\widehat{D}_{4}\widehat{D}_{2};$

$(10)$ $\widehat{D}_{2}\widehat{D}_{3}\widehat{D}_{4}-\widehat{D}_{4}\widehat{D}_{3}\widehat{D}_{2}-\widehat{D}_{2}\widehat{D}_{4}\widehat{D}_{1}+
\widehat{D}_{1}\widehat{D}_{4}\widehat{D}_{2};$

$(11)$ $\widehat{D}_{1}\widehat{D}_{3}-\widehat{D}_{3}\widehat{D}_{1};$

$(12)$ $\widehat{D}_{2}\widehat{D}_{4}-\widehat{D}_{4}\widehat{D}_{2};$

$(13)$ $\widehat{D}_{1}\widehat{D}_{2}\widehat{D}_{3}\widehat{D}_{4}-\widehat{D}_{4}\widehat{D}_{3}\widehat{D}_{2}\widehat{D}_{1}.$

By Lemma \ref{b2-lemma-relations}, it is a direct verification that $J$ is in fact generated by the following two elements $-\widehat{D}_{2}+\widehat{D}_{2}\widehat{D}_{3}-\widehat{D}_{3}-\widehat{D}_{4}\widehat{D}_{1}+
\widehat{D}_{3}\widehat{D}_{4}+\widehat{D}_{1}\widehat{D}_{2}$ and $\widehat{D}_{1}\widehat{D}_{2}\widehat{D}_{3}\widehat{D}_{4}-\widehat{D}_{4}\widehat{D}_{3}\widehat{D}_{2}\widehat{D}_{1}.$

Finally, note that
\begin{align*}
-\widehat{D}_{2}+\widehat{D}_{2}\widehat{D}_{3}-\widehat{D}_{3}-\widehat{D}_{4}\widehat{D}_{1}+
\widehat{D}_{3}\widehat{D}_{4}+\widehat{D}_{1}\widehat{D}_{2}=Q_{12}^{(1, 1, 1)}
\end{align*}
and
\begin{align*}
\widehat{D}_{1}\widehat{D}_{2}\widehat{D}_{3}\widehat{D}_{4}-\widehat{D}_{4}\widehat{D}_{3}\widehat{D}_{2}\widehat{D}_{1}=R_{12}^{(1, 1, 4)}
\end{align*}
in the notation of [BK, Subsection 7.4]. By [BK, Proposition 7.20] it is easy to see that the two elements $Q_{12}^{(1, 1, 1)}$ and $R_{12}^{(1, 1, 4)}$ belong to $K_{12}(W_1).$

Moreover, if $-\widehat{D}_{2}+\widehat{D}_{2}\widehat{D}_{3}-\widehat{D}_{3}-\widehat{D}_{4}\widehat{D}_{1}+
\widehat{D}_{3}\widehat{D}_{4}+\widehat{D}_{1}\widehat{D}_{2}=0,$ that is, \eqref{B-relation5} holds, it is easy to check that $\widehat{D}_1\widehat{D}_2\widehat{D}_3\widehat{D}_4=\widehat{D}_1\widehat{D}_4\widehat{D}_1\widehat{D}_4=D_1D_2D_1D_2$ and $\widehat{D}_4\widehat{D}_3\widehat{D}_2\widehat{D}_1=\widehat{D}_4\widehat{D}_1\widehat{D}_4\widehat{D}_1=D_2D_1D_2D_1.$ We are done.
\end{proof}

From the presentation in Theorem \ref{theoremb1}, we can get the following relations in $\mathbf{H}(W_1)$.
\begin{lemma}\label{b2-lemma-relations} We have, in $\mathbf{H}(W_1),$
\begin{align}
D_1s_2s_1D_2&=D_2D_1s_2s_1+s_2D_1D_2s_1+s_2s_1D_2D_1+s_2s_1D_2s_1+s_2D_1s_2s_1,\label{B-relation7}\\[0.1em]
D_1D_2D_1s_2&+D_1s_2D_1D_2=D_2D_1s_2D_1+s_2D_1D_2D_1;\label{B-relation8}\\[0.1em]
D_2D_1D_2s_1&+D_2s_1D_2D_1=D_1D_2s_1D_2+s_1D_2D_1D_2.\label{B-relation9}\\[0.1em]
D_1s_2D_1s_2&=s_2D_1s_2D_1;\label{B-relation10}\\[0.1em]
D_2s_1D_2s_1&=s_1D_2s_1D_2;\label{B-relation11}
\end{align}
\end{lemma}
\begin{proof}
\eqref{B-relation7} From \eqref{B-relation5} we can get that \begin{align*}s_1D_2s_1s_2D_1s_1=s_1D_1D_2s_1s_2s_1+D_2D_1s_2s_1+s_2D_1D_2s_1+s_2D_1s_2s_1+D_2s_1s_2s_1.\end{align*}
By $D_2s_1s_2s_1=s_1s_2s_1D_2,$ $s_1D_2s_1s_2=s_2s_1D_2s_1,$ $s_1D_2s_1s_2D_1s_1=-s_2s_1D_2D_1+s_2s_1D_2-s_2s_1D_2s_1$ and $s_1D_1D_2s_1s_2s_1=-D_1s_2s_1D_2+s_2s_1D_2-D_2s_1s_2s_1$, we can easily get \eqref{B-relation7}.

\eqref{B-relation8} From \eqref{B-relation7} we can get that
$$D_1s_2s_1D_2=D_1D_2D_1s_2s_1+D_1s_2D_1D_2s_1+D_1s_2s_1D_2D_1+D_1s_2s_1D_2s_1+D_1s_2D_1s_2s_1.$$ So we have
\begin{align*}
D_1D_2D_1s_2+D_1s_2D_1D_2+&D_1s_2s_1D_2D_1s_1+D_1s_2s_1D_2+D_1s_2D_1s_2\\
&=D_2D_1s_2+s_2D_1D_2+s_2s_1D_2D_1s_1+s_2s_1D_2+s_2D_1s_2.
\end{align*}

From \eqref{B-relation5} we can get that
\begin{align*}
D_2s_1s_2D_1=D_1D_2s_1s_2D_1+s_1D_2D_1s_2D_1+s_1s_2D_1D_2D_1+s_1s_2D_1s_2D_1+s_1D_2s_1s_2D_1.
\end{align*}
So we have
\begin{align*}
D_2D_1s_2D_1+s_2D_1D_2D_1+&s_1D_1D_2s_1s_2D_1+s_2D_1s_2D_1+D_2s_1s_2D_1\\
&=s_1D_1D_2s_1s_2+D_2D_1s_2+s_2D_1D_2+s_2D_1s_2+D_2s_1s_2.
\end{align*}
In order to show \eqref{B-relation8}, it suffices to show that
\begin{align*}
D_1s_2s_1D_2D_1s_1+D_1s_2s_1D_2+&s_1D_1D_2s_1s_2+D_2s_1s_2\\
&=s_2s_1D_2D_1s_1+s_2s_1D_2+s_1D_1D_2s_1s_2D_1+D_2s_1s_2D_1.
\end{align*}

The left-hand side is equal to $-D_1s_2s_1D_2s_1D_1+D_1s_2s_1D_2s_1-D_1s_1D_2s_1s_2+s_1D_2s_1s_2=-D_1s_2s_1D_2s_1D_1+s_1D_2s_1s_2,$ and the right-hand side is equal to $-s_2s_1D_2s_1D_1+s_2s_1D_2s_1-D_1s_1D_2s_1s_2D_1+s_1D_2s_1s_2D_1=s_2s_1D_2s_1-D_1s_1D_2s_1s_2D_1.$ We are done.

\eqref{B-relation9} It can be proved similarly.

\eqref{B-relation10} From \eqref{B-relation5} we can get that
$$s_1D_2s_1s_2D_1=s_1D_1D_2s_1s_2+D_2D_1s_2+s_2D_1D_2+s_2D_1s_2+D_2s_1s_2.$$
From this we have
\begin{align*}
D_1s_1D_2s_1s_2D_1=D_1s_1D_1D_2s_1s_2+D_1D_2D_1s_2+D_1s_2D_1D_2+D_1s_2D_1s_2+D_1D_2s_1s_2,
\end{align*}
and
\begin{align*}
s_1D_2s_1s_2D_1=s_1D_1D_2s_1s_2D_1+D_2D_1s_2D_1+s_2D_1D_2D_1+s_2D_1s_2D_1+D_2s_1s_2D_1.
\end{align*}
Note that $s_1D_1+D_1s_1=s_1-1$ and $D_1s_1D_1=-D_1$. By \eqref{B-relation8} we can easily get that \eqref{B-relation10} holds.

\eqref{B-relation11} It can be proved similarly.
\end{proof}

Set $\alpha_{1}^{\vee} :=\varepsilon_{1}^{\vee}-\varepsilon_{2}^{\vee},$ $\alpha_{2}^{\vee} :=2\varepsilon_{2}^{\vee}$ and also $\alpha_1 :=\varepsilon_1-\varepsilon_2,$ $\alpha_2 :=\varepsilon_2,$ where $\varepsilon_1$ and $\varepsilon_2$ are a set of bases of the Euclidean space $\mathbb{R}^{2}$ and $\varepsilon_1^{\vee}$ and $\varepsilon_2^{\vee}$ are the dual bases. We have
\begin{align*}
a_{11} :=\langle\alpha_{1}^{\vee},\alpha_{1}\rangle=2,\quad a_{22} :=\langle\alpha_{2}^{\vee},\alpha_{2}\rangle=2,\quad a_{12} :=\langle\alpha_{2}^{\vee},\alpha_{1}\rangle=-2,\quad a_{21} :=\langle\alpha_{1}^{\vee},\alpha_{2}\rangle=-1.
\end{align*}
Let $\mathcal{Q}_1$ be the field of fractions of the Laurent polynomial ring $\mathcal{L}_1=\mathbb{Z}[t_1^{\pm 1}, t_2^{\pm 1}].$ We define the action of $W_1$ on $\mathcal{Q}_1$ by
\begin{align*}
s_1(t_1)=t_1^{-1},\quad s_1(t_2)=t_1^{-a_{12}}t_2=t_1^{2}t_2,\quad s_2(t_1)=t_2^{-a_{21}}t_1=t_1t_2,\quad s_2(t_2)=t_2^{-1}.
\end{align*}

Recall from [BK, Proposition 7.42] that the assignments $D_{i} \mapsto \frac{1}{1-t_{i}}(1-s_{i})$ and $s_{i}\mapsto s_{i}$ $(1\leq i\leq2),$ define a homomorphism of algebras $\hat{p}_{W_1} :\hat{\bf{H}}(W_{1})\rightarrow \mathcal{Q}_1\rtimes \mathbb{Z}W_{1}.$ Then we have the following lemma.
\begin{lemma}\label{b2-lemma1}
We have

$(1)$ $\hat{p}_{W_1}(\widehat{D}_2+\widehat{D}_3+\widehat{D}_4\widehat{D}_1-\widehat{D}_1\widehat{D}_2-\widehat{D}_2\widehat{D}_3-\widehat{D}_3\widehat{D}_4)=0.$

$(2)$ $\hat{p}_{W_1}(\widehat{D}_1\widehat{D}_2\widehat{D}_3\widehat{D}_4-\widehat{D}_4\widehat{D}_3\widehat{D}_2\widehat{D}_1)=0.$
\end{lemma}
\begin{proof}
$(1)$ Set $\tau_1 :=\frac{1}{1-t_1}$ and $\tau_2 :=\frac{1}{1-t_2}$ such that $\hat{p}_{W_1}(\widehat{D}_1)=\tau_1(1-s_1)$ and $\hat{p}_{W_1}(\widehat{D}_4)=\tau_2(1-s_2).$ Therefore, we have
\begin{align*}
\hat{p}_{W_1}(\widehat{D}_2)=s_1\tau_2(1-s_2)s_1=\tau_{12}(1-s_{12}),\quad \mathrm{and}\quad \hat{p}_{W_1}(\widehat{D}_3)=s_2\tau_1(1-s_1)s_2=\tau_{21}(1-s_{21}),
\end{align*}
where $\tau_{12}=s_1\tau_{2}s_1=\frac{1}{1-t_1^{2}t_2},$ $\tau_{21}=s_2\tau_{1}s_2=\frac{1}{1-t_1t_2},$ $s_{12}=s_1s_2s_1$ and $s_{21}=s_2s_1s_2.$

Thus, we get that
\begin{align*}
\hat{p}_{W_1}(\widehat{D}_4\widehat{D}_1)=\tau_2(1-s_2)\tau_1(1-s_1)=\tau_2\tau_1(1-s_1)-\tau_2\tau_{21}s_2(1-s_1);
\end{align*}
\begin{align*}
\hat{p}_{W_1}(\widehat{D}_1\widehat{D}_2)=\tau_1(1-s_1)\tau_{12}(1-s_{12})=\tau_1\tau_{12}(1-s_{12})-\tau_1\tau_{2}s_1(1-s_{12});
\end{align*}
\begin{align*}
\hat{p}_{W_1}(\widehat{D}_2\widehat{D}_3)=\tau_{12}(1-s_{12})\tau_{21}(1-s_{21})=\tau_{12}\tau_{21}(1-s_{21})-\tau_{12}(1-\tau_{1})s_{12}(1-s_{21});
\end{align*}
\begin{align*}
\hat{p}_{W_1}(\widehat{D}_3\widehat{D}_4)=\tau_{21}(1-s_{21})\tau_{2}(1-s_{2})=\tau_{21}\tau_{2}(1-s_{2})-\tau_{21}(1-\tau_{12})s_{21}(1-s_{2});
\end{align*}
Therefore, we have
\begin{align}
&\hat{p}_{W_1}(\widehat{D}_2+\widehat{D}_3+\widehat{D}_4\widehat{D}_1-\widehat{D}_1\widehat{D}_2-\widehat{D}_2\widehat{D}_3-\widehat{D}_3\widehat{D}_4)
\notag\\=&\tau_2\tau_1-\tau_1\tau_{12}-\tau_{12}\tau_{21}-\tau_{21}\tau_2+\tau_{12}+\tau_{21}-\tau_2\tau_1s_1-\tau_2\tau_{21}s_2+
\tau_2\tau_{21}s_2s_1\notag\\
&+\tau_1\tau_{12}s_{12}+\tau_1\tau_2s_1-\tau_1\tau_2s_2s_1+\tau_{12}\tau_{21}s_{21}+\tau_{12}(1-\tau_1)s_{12}-\tau_{12}(1-\tau_1)s_2s_1\notag\\
&+\tau_{21}\tau_2s_2+\tau_{21}(1-\tau_{12})s_{21}-\tau_{21}(1-\tau_{12})s_{2}s_1-\tau_{12}s_{12}-\tau_{21}s_{21}.
\label{b2-lemma-equatlity1}
\end{align}
It is easy to check that $\tau_2\tau_1-\tau_1\tau_{12}-\tau_{12}\tau_{21}-\tau_{21}\tau_2+\tau_{12}+\tau_{21}=0.$ By this and \eqref{b2-lemma-equatlity1}, we can see that $(1)$ holds.

$(2)$ We can easily deduce that $\widehat{D}_1\widehat{D}_2\widehat{D}_3\widehat{D}_4=\widehat{D}_1\widehat{D}_4\widehat{D}_1\widehat{D}_4$ by \eqref{B-relation5}. Thus, by $(1)$, in order to check $(2),$ it suffices to prove that $\hat{p}_{W_1}(\widehat{D}_1\widehat{D}_4\widehat{D}_1\widehat{D}_4-\widehat{D}_4\widehat{D}_1\widehat{D}_4\widehat{D}_1)=0.$

We have
\begin{align}\label{b2-lemma-equatlity2}
&\hat{p}_{W_1}(\widehat{D}_1\widehat{D}_4\widehat{D}_1\widehat{D}_4-\widehat{D}_4\widehat{D}_1\widehat{D}_4\widehat{D}_1)\notag\\
=&\tau_1(1-s_1)\tau_2(1-s_2)\tau_1(1-s_1)\tau_2(1-s_2)-\tau_2(1-s_2)\tau_1(1-s_1)\tau_2(1-s_2)\tau_1(1-s_1)\notag\\
=&\big(\tau_1\tau_2(1-s_2)-\tau_1\tau_{12}s_1(1-s_2)\big)\big(\tau_1\tau_2(1-s_2)-\tau_1\tau_{12}s_1(1-s_2)\big)\notag\\
&-\big(\tau_2\tau_1(1-s_1)-\tau_2\tau_{21}s_2(1-s_1)\big)\big(\tau_2\tau_1(1-s_1)-\tau_2\tau_{21}s_2(1-s_1)\big)\notag\\
=&\tau_1\tau_2\tau_1\tau_2(1-s_2)+\tau_1\tau_2\tau_{21}(1-\tau_2)(1-s_2)-\tau_1\tau_2\tau_1\tau_{12}s_1(1-s_2)\notag\\
&+\tau_1\tau_2\tau_{21}\tau_{12}s_2s_1(1-s_2)-\tau_1\tau_{12}(1-\tau_1)\tau_{12}s_1(1-s_2)+\tau_1\tau_{12}\tau_{21}(1-\tau_{12})s_1s_2(1-s_2)\notag\\
&+\tau_1\tau_{12}(1-\tau_1)\tau_2(1-s_2)-\tau_1\tau_{12}\tau_{21}\tau_2s_1s_2s_1(1-s_2)\notag\\
&-\tau_2\tau_1\tau_2\tau_1(1-s_1)-\tau_2\tau_1\tau_{12}(1-\tau_1)(1-s_1)+\tau_2\tau_1\tau_2\tau_{21}s_2(1-s_1)\notag\\
&-\tau_2\tau_1\tau_{12}\tau_{21}s_1s_2(1-s_1)+\tau_2\tau_{21}(1-\tau_2)\tau_{21}s_2(1-s_1)-\tau_2\tau_{21}\tau_{12}(1-\tau_{21})s_2s_1(1-s_1)\notag\\
&-\tau_2\tau_{21}(1-\tau_2)\tau_1(1-s_1)-\tau_2\tau_{21}\tau_{12}\tau_1s_2s_1s_2(1-s_1).
\end{align}
It is easy to see that on the right-hand side of \eqref{b2-lemma-equatlity2}, the constant coefficient is equal to
\begin{align*}
\tau_1\tau_2\tau_1\tau_2+\tau_1\tau_2\tau_{21}(1-\tau_2)+\tau_1\tau_{12}(1-\tau_1)\tau_2-\tau_2\tau_1\tau_2\tau_1-
\tau_2\tau_1\tau_{12}(1-\tau_1)-\tau_2\tau_{21}(1-\tau_2)\tau_1=0;
\end{align*}
the coefficient of $s_1$ is equal to
\begin{align*}
&-\tau_1\tau_2\tau_1\tau_{12}-\tau_1\tau_{12}(1-\tau_1)\tau_{12}-\tau_1\tau_{12}\tau_{21}(1-\tau_{12})\notag\\
&\qquad \qquad +\tau_2\tau_1\tau_2\tau_1+\tau_2\tau_1\tau_{12}(1-\tau_1)+\tau_2\tau_{21}(1-\tau_2)\tau_1\notag\\
=&\tau_1\tau_2(\tau_1\tau_2+\tau_{12}+\tau_{21}-\tau_1\tau_{12}-\tau_2\tau_{21})+
\tau_1\tau_{12}(-\tau_{12}+\tau_1\tau_{12}+\tau_{12}\tau_{21}-\tau_{21}-\tau_1\tau_2)\notag\\
=&\tau_1\tau_2\tau_{12}\tau_{21}-\tau_1\tau_{12}\tau_{21}\tau_2=0;
\end{align*}
similarly, the coefficient of $s_2$ can be proved to be zero; the coefficient of $s_1s_2$ is equal to
\begin{align*}
&\tau_1\tau_2\tau_1\tau_{12}+\tau_1\tau_{12}(1-\tau_1)\tau_{12}+\tau_1\tau_{12}\tau_{21}(1-\tau_{12})-\tau_2\tau_1\tau_{12}\tau_{21}\notag\\
=&\tau_1\tau_{12}(\tau_1\tau_{2}+\tau_{12}-\tau_{1}\tau_{12}+\tau_{21}-\tau_{12}\tau_{21}-\tau_2\tau_{21})=0;
\end{align*}
similarly, the coefficient of $s_2s_1$ can be proved to be zero; the coefficient of $s_1s_2s_1$ is equal to $-\tau_1\tau_{12}\tau_{21}\tau_{2}+\tau_2\tau_1\tau_{12}\tau_{21}=0;$ similarly, the coefficient of $s_2s_1s_2$ and $s_1s_2s_1s_2$ can be proved to be zero. Thus, we can see that the right-hand side of \eqref{b2-lemma-equatlity2} is equal to zero, that is, $(2)$ is proved.
\end{proof}

Lemma \ref{b2-lemma1} immediately yields the following corollary.
\begin{corollary}\label{b2-corollary-module}
The assignments $D_{i} \mapsto \frac{1}{1-t_{i}}(1-s_{i})$ and $s_{i}\mapsto s_{i}$ $(1\leq i\leq2),$ define a homomorphism of algebras $p_{W_1} :\mathbf{H}(W_1)\rightarrow \mathcal{Q}_1\rtimes \mathbb{Z}W_{1}.$
\end{corollary}

We have a natural action of $\mathcal{Q}_1\rtimes \mathbb{Z}W_{1}$ on $\mathcal{Q}_1$ by $(tw)(t')=t\cdot w(t')$ for $t, t'\in \mathcal{Q}_1$ and $w\in W_{1}.$ This, composing with $\hat{p}_{W_1}$, gives an action of $\hat{\bf{H}}(W_{1})$ on $\mathcal{Q}_1$ such that $\mathcal{L}_1$ is invariant and $\mathcal{Q}_1$ and $\mathcal{L}_1$ are both module algebras over $\hat{\bf{H}}(W_{1}).$ By Corollary \ref{b2-corollary-module} we get that $p_{W_1}$ defines a structure of a module algebra of $\mathbf{H}(W_1)$ on $\mathcal{Q}_1$ such that $\mathcal{L}_1$ is a module subalgebra. Thus, we verify Conjecture \ref{introduction-conjecture1} for the Coxeter group $W_{1}.$

From the proof of Theorem \ref{theoremb1}, it is easy to see that Conjecture \ref{introduction-conjecture2} also holds for the Coxeter group $W_{1}.$

\section{Type $G_2$}
Let $W_2$ be the Coxeter group of type $G_2$ with generators $s_1,s_2$ and relations $s_1^{2}=s_2^{2}=1$ and $s_1s_2s_1s_2s_1s_2=s_2s_1s_2s_1s_2s_1.$ We then have the following theorem.

\begin{theorem}\label{theoremg2}
The Hecke-Hopf algebra $\mathbf{H}(W_2),$ associated to $W_2,$ has the following presentation$:$ it is generated by the elements $s_1,s_2,D_1,D_2$ with the following relations$:$
\begin{align}
s_1^{2}=s_2^{2}&=1, \quad s_1s_2s_1s_2s_1s_2=s_2s_1s_2s_1s_2s_1;\label{G-relation1}\\[0.1em]
D_{i}^{2}=&D_{i};\label{G-relation2}\\[0.1em]
s_{i}D_i+&D_{i}s_i=s_i-1;\label{G-relation3}\\[0.1em]
D_1s_2s_1&s_2s_1s_2=s_2s_1s_2s_1s_2D_1,\quad D_2s_1s_2s_1s_2s_1=s_1s_2s_1s_2s_1D_2;\label{G-relation4}\\[0.1em]
D_1s_2s_1&s_2s_1D_2=D_2D_1s_2s_1s_2s_1+s_2D_1D_2s_1s_2s_1+s_2s_1D_2D_1s_2s_1\notag\\
&\qquad \qquad +s_2s_1s_2D_1D_2s_1+s_2s_1s_2s_1D_2D_1+s_2s_1D_2s_1s_2s_1\notag\\
&\qquad \qquad +s_2s_1s_2s_1D_2s_1+s_2D_1s_2s_1s_2s_1+s_2s_1s_2D_1s_2s_1;\label{G-relation5}\\[0.1em]
D_2s_1s_2&D_1s_2s_1=s_1s_2D_1s_2s_1D_2;\label{G-relation6}\\[0.1em]
D_1s_2s_1&s_2D_1s_2=s_2s_1s_2D_1s_2D_1+s_2D_1s_2D_1s_2s_1+s_2s_1s_2D_1s_2s_1;\label{G-relation7}\\[0.1em]
D_2s_1s_2&s_1D_2s_1=s_1s_2s_1D_2s_1D_2+s_1D_2s_1D_2s_1s_2+s_1s_2s_1D_2s_1s_2;\label{G-relation8}\\[0.1em]
D_2D_1&D_2s_1s_2D_1+D_2s_1D_2D_1s_2D_1+D_2s_1D_2s_1s_2D_1\notag\\
&\qquad\qquad\qquad +D_2s_1s_2D_1D_2D_1+D_2s_1s_2D_1s_2D_1=\notag\\
&D_1D_2D_1D_2s_1s_2+D_1D_2s_1D_2D_1s_2+D_1D_2s_1D_2s_1s_2+D_1D_2s_1s_2D_1D_2\notag\\
&+D_1D_2s_1s_2D_1s_2+s_1D_2D_1D_2D_1s_2+s_1D_2D_1D_2s_1s_2+s_1D_2D_1s_2D_1D_2\notag\\
&+s_1D_2D_1s_2D_1s_2+s_1D_2s_1D_2D_1s_2+s_1D_2s_1D_2s_1s_2+s_1D_2s_1s_2D_1D_2\notag\\
&+s_1D_2s_1s_2D_1s_2+s_1s_2D_1D_2D_1D_2+s_1s_2D_1D_2D_1s_2+s_1s_2D_1s_2D_1D_2\notag\\
&+s_1s_2D_1s_2D_1s_2;\label{G-relation9}\\[0.1em]
s_1D_2s_1&D_2D_1D_2+s_1D_2D_1D_2s_1D_2+D_1D_2s_1D_2s_1D_2=\notag\\
&\qquad \qquad \qquad D_2s_1D_2s_1D_2D_1+D_2s_1D_2D_1D_2s_1+D_2D_1D_2s_1D_2s_1;\label{G-relation10}\\[0.1em]
s_2D_1s_2&D_1D_2D_1+s_2D_1D_2D_1s_2D_1+D_2D_1s_2D_1s_2D_1=\notag\\
&\qquad \qquad \qquad D_1s_2D_1s_2D_1D_2+D_1s_2D_1D_2D_1s_2+D_1D_2D_1s_2D_1s_2;\label{G-relation11}\\[0.1em]
D_1D_2&D_1D_2D_1D_2=D_2D_1D_2D_1D_2D_1.\label{G-relation12}
\end{align}
\end{theorem}
\begin{proof}
The proof of this theorem is exactly the same as that of Theorem \ref{theoremb1}, but much more complicated. We skip the details.
\end{proof}

From the presentation in Theorem \ref{theoremg2}, we can get the following relations in $\mathbf{H}(W_2)$.
\begin{lemma} We have, in $\mathbf{H}(W_2),$
\begin{align}
D_2s_1s_2s_1s_2D_1=&D_1D_2s_1s_2s_1s_2+s_1D_2D_1s_2s_1s_2+s_1s_2D_1D_2s_1s_2\notag\\
&+s_1s_2s_1D_2D_1s_2+s_1s_2s_1s_2D_1D_2+s_1s_2D_1s_2s_1s_2\notag\\
&+s_1s_2s_1s_2D_1s_2+s_1D_2s_1s_2s_1s_2+s_1s_2s_1D_2s_1s_2.\label{G-relation13}
\end{align}
\begin{align}
&D_2D_1D_2s_1s_2s_1+D_2s_1D_2D_1s_2s_1+D_2s_1D_2s_1s_2s_1\notag\\
&\qquad \qquad +D_2s_1s_2D_1D_2s_1+D_2s_1s_2s_1D_2D_1+D_2s_1s_2s_1D_2s_1\notag\\
=&s_1s_2s_1D_2D_1D_2+s_1s_2D_1D_2s_1D_2+s_1s_2s_1D_2s_1D_2\notag\\
&\qquad \qquad+s_1D_2D_1s_2s_1D_2+D_1D_2s_1s_2s_1D_2+s_1D_2s_1s_2s_1D_2.\label{G-relation14}
\end{align}
\begin{align}
&D_1D_2D_1s_2s_1s_2+D_1s_2D_1D_2s_1s_2+D_1s_2D_1s_2s_1s_2\notag\\
&\qquad \qquad+D_1s_2s_1D_2D_1s_2+D_1s_2s_1s_2D_1D_2+D_1s_2s_1s_2D_1s_2\notag\\
=&s_2s_1s_2D_1D_2D_1+s_2s_1D_2D_1s_2D_1+s_2s_1s_2D_1s_2D_1\notag\\
&\qquad \qquad+s_2D_1D_2s_1s_2D_1+D_2D_1s_2s_1s_2D_1+s_2D_1s_2s_1s_2D_1.\label{G-relation15}
\end{align}
\begin{align}
&D_2D_1D_2s_1s_2s_1+D_2D_1s_2s_1D_2s_1+D_2s_1D_2D_1s_2s_1\notag\\
&\qquad+D_2s_1D_2s_1s_2D_1+D_2s_1s_2D_1D_2s_1+D_2s_1s_2s_1D_2D_1\notag\\
&\qquad\qquad+s_2D_1D_2s_1D_2s_1+s_2s_1D_2D_1D_2s_1+s_2s_1D_2s_1D_2D_1\notag\\
=&s_1s_2s_1D_2D_1D_2+s_1s_2D_1D_2s_1D_2+s_1D_2s_1s_2D_1D_2\notag\\
&\qquad+D_1s_2s_1D_2s_1D_2+s_1D_2D_1s_2s_1D_2+D_1D_2s_1s_2s_1D_2\notag\\
&\qquad\qquad+s_1D_2s_1D_2D_1s_2+s_1D_2D_1D_2s_1s_2+D_1D_2s_1D_2s_1s_2.\label{G-relation16}
\end{align}
\begin{align}
&D_1D_2D_1s_2s_1s_2+D_1D_2s_1s_2D_1s_2+D_1s_2D_1D_2s_1s_2\notag\\
&\qquad+D_1s_2D_1s_2s_1D_2+D_1s_2s_1D_2D_1s_2+D_1s_2s_1s_2D_1D_2\notag\\
&\qquad\qquad+s_1D_2D_1s_2D_1s_2+s_1s_2D_1D_2D_1s_2+s_1s_2D_1s_2D_1D_2\notag\\
=&s_2s_1s_2D_1D_2D_1+s_2s_1D_2D_1s_2D_1+s_2D_1s_2s_1D_2D_1\notag\\
&\qquad+D_2s_1s_2D_1s_2D_1+s_2D_1D_2s_1s_2D_1+D_2D_1s_2s_1s_2D_1\notag\\
&\qquad\qquad+s_2D_1s_2D_1D_2s_1+s_2D_1D_2D_1s_2s_1+D_2D_1s_2D_1s_2s_1.\label{G-relation17}
\end{align}
\begin{align}
D_1s_2D_1s_2D_1s_2=s_2D_1s_2D_1s_2D_1.\label{G-relation18}
\end{align}
\begin{align}
D_2s_1D_2s_1D_2s_1=s_1D_2s_1D_2s_1D_2.\label{G-relation19}
\end{align}
\begin{align}
D_1D_2&D_1s_2s_1D_2+D_1s_2D_1D_2s_1D_2+D_1s_2D_1s_2s_1D_2\notag\\
&\qquad\qquad\qquad +D_1s_2s_1D_2D_1D_2+D_1s_2s_1D_2s_1D_2=\notag\\
&D_2D_1D_2D_1s_2s_1+D_2D_1s_2D_1D_2s_1+D_2D_1s_2D_1s_2s_1+D_2D_1s_2s_1D_2D_1\notag\\
&+D_2D_1s_2s_1D_2s_1+s_2D_1D_2D_1D_2s_1+s_2D_1D_2D_1s_2s_1+s_2D_1D_2s_1D_2D_1\notag\\
&+s_2D_1D_2s_1D_2s_1+s_2D_1s_2D_1D_2s_1+s_2D_1s_2D_1s_2s_1+s_2D_1s_2s_1D_2D_1\notag\\
&+s_2D_1s_2s_1D_2s_1+s_2s_1D_2D_1D_2D_1+s_2s_1D_2D_1D_2s_1+s_2s_1D_2s_1D_2D_1\notag\\
&+s_2s_1D_2s_1D_2s_1;\label{G-relation20}
\end{align}
\begin{align}
&s_1s_2D_1D_2D_1D_2+s_1D_2D_1s_2D_1D_2+s_1D_2D_1D_2D_1s_2\notag\\
&\qquad+D_1s_2s_1D_2D_1D_2+D_1s_2D_1D_2s_1D_2+D_1D_2s_1s_2D_1D_2\notag\\
&\qquad\qquad+D_1D_2s_1D_2D_1s_2+D_1D_2D_1s_2s_1D_2+D_1D_2D_1D_2s_1s_2\notag\\
=&s_2s_1D_2D_1D_2D_1+s_2D_1D_2s_1D_2D_1+s_2D_1D_2D_1D_2s_1\notag\\
&\qquad+D_2s_1s_2D_1D_2D_1+D_2s_1D_2D_1s_2D_1+D_2D_1s_2s_1D_2D_1\notag\\
&\qquad\qquad+D_2D_1s_2D_1D_2s_1+D_2D_1D_2s_1s_2D_1+D_2D_1D_2D_1s_2s_1;\label{G-relation21}
\end{align}
\begin{align}
&s_1D_2D_1D_2D_1D_2+D_1D_2s_1D_2D_1D_2+D_1D_2D_1D_2s_1D_2\notag\\
&\qquad\quad=D_2s_1D_2D_1D_2D_1+D_2D_1D_2s_1D_2D_1+D_2D_1D_2D_1D_2s_1;\label{G-relation22}
\end{align}
\begin{align}
&D_1s_2D_1D_2D_1D_2+D_1D_2D_1s_2D_1D_2+D_1D_2D_1D_2D_1s_2\notag\\
&\qquad\quad=s_2D_1D_2D_1D_2D_1+D_2D_1s_2D_1D_2D_1+D_2D_1D_2D_1s_2D_1.\label{G-relation23}
\end{align}
\end{lemma}
\begin{proof}
\eqref{G-relation13} From \eqref{G-relation5} we can easily get that
\begin{align*}
s_2D_1s_2s_1s_2s_1D_2s_2=&s_2D_2D_1s_2s_1s_2s_1s_2+D_1D_2s_1s_2s_1s_2+s_1D_2D_1s_2s_1s_2\\
&+s_1s_2D_1D_2s_1s_2+s_1s_2s_1D_2D_1s_2+s_1D_2s_1s_2s_1s_2\\
&+s_1s_2s_1D_2s_1s_2+D_1s_2s_1s_2s_1s_2+s_1s_2D_1s_2s_1s_2.
\end{align*}
From this, we can easily get that \eqref{G-relation13} holds.

\eqref{G-relation14} From \eqref{G-relation13} we get that
\begin{align*}
D_2s_1s_2s_1s_2D_1=&D_2D_1D_2s_1s_2s_1s_2+D_2s_1D_2D_1s_2s_1s_2+D_2s_1s_2D_1D_2s_1s_2\\
&+D_2s_1s_2s_1D_2D_1s_2+D_2s_1s_2s_1s_2D_1D_2+D_2s_1s_2D_1s_2s_1s_2\\
&+D_2s_1s_2s_1s_2D_1s_2+D_2s_1D_2s_1s_2s_1s_2+D_2s_1s_2s_1D_2s_1s_2.
\end{align*}
Thus, we have
\begin{align*}
&D_2D_1D_2s_1s_2s_1+D_2s_1D_2D_1s_2s_1+D_2s_1s_2D_1D_2s_1\\
&\qquad +D_2s_1s_2s_1D_2D_1+D_2s_1s_2s_1s_2D_1D_2s_2+D_2s_1s_2D_1s_2s_1\\
&\qquad \qquad+D_2s_1s_2s_1s_2D_1+D_2s_1D_2s_1s_2s_1+D_2s_1s_2s_1D_2s_1\\
=&D_1D_2s_1s_2s_1+s_1D_2D_1s_2s_1+s_1s_2D_1D_2s_1\\
&\qquad +s_1s_2s_1D_2D_1+s_1s_2s_1s_2D_1D_2s_2+s_1s_2D_1s_2s_1\\
&\qquad \qquad+s_1s_2s_1s_2D_1+s_1D_2s_1s_2s_1+s_1s_2s_1D_2s_1.
\end{align*}
Similarly, we have
\begin{align*}
&s_2D_2D_1s_2s_1s_2s_1+D_1D_2s_1s_2s_1+s_1D_2D_1s_2s_1\\
&\qquad+s_1s_2D_1D_2s_1+s_1s_2s_1D_2D_1+s_1D_2s_1s_2s_1\\
&\qquad \qquad+s_1s_2s_1D_2s_1+D_1s_2s_1s_2s_1+s_1s_2D_1s_2s_1\\
=&s_2D_2D_1s_2s_1s_2s_1D_2+D_1D_2s_1s_2s_1D_2+s_1D_2D_1s_2s_1D_2\\
&\qquad+s_1s_2D_1D_2s_1D_2+s_1s_2s_1D_2D_1D_2+s_1D_2s_1s_2s_1D_2\\
&\qquad \qquad+s_1s_2s_1D_2s_1D_2+D_1s_2s_1s_2s_1D_2+s_1s_2D_1s_2s_1D_2.
\end{align*}
Thus, in order to prove that \eqref{G-relation14} holds, it suffices to show that
\begin{align*}
&D_2s_1s_2s_1s_2D_1D_2s_2+D_2s_1s_2s_1s_2D_1+s_2D_2D_1s_2s_1s_2s_1+D_1s_2s_1s_2s_1\\
&\qquad\qquad=s_1s_2s_1s_2D_1D_2s_2+s_1s_2s_1s_2D_1+s_2D_2D_1s_2s_1s_2s_1D_2+D_1s_2s_1s_2s_1D_2.
\end{align*}
It is easy to see that the left-hand side and the right-hand side are both equal to
$$-D_2s_1s_2s_1s_2D_1s_2D_2+s_2D_1s_2s_1s_2s_1.$$

\eqref{G-relation15} It can be proved similarly.

\eqref{G-relation16} By \eqref{G-relation14}, in order to prove that \eqref{G-relation16} holds, it suffices to show that
\begin{align}\label{G2-equality1}
&s_1D_2s_1s_2D_1D_2+s_1D_2s_1D_2D_1s_2+s_1D_2D_1D_2s_1s_2+D_1s_2s_1D_2s_1D_2\notag\\
&\qquad+D_1D_2s_1D_2s_1s_2+D_2s_1D_2s_1s_2s_1+D_2s_1s_2s_1D_2s_1\notag\\
=&s_1D_2s_1s_2s_1D_2+s_1s_2s_1D_2s_1D_2+s_2s_1D_2s_1D_2D_1+s_2s_1D_2D_1D_2s_1\notag\\
&\qquad+s_2D_1D_2s_1D_2s_1+D_2s_1D_2s_1s_2D_1+D_2D_1s_2s_1D_2s_1.
\end{align}
By \eqref{G-relation8}, we see that, in order to show \eqref{G2-equality1}, it suffices to prove that
\begin{align}\label{G2-equality2}
&s_1D_2s_1s_2D_1D_2+s_1D_2s_1D_2D_1s_2+s_1D_2D_1D_2s_1s_2\notag\\
&\qquad+D_1s_2s_1D_2s_1D_2+D_1D_2s_1D_2s_1s_2+s_1D_2s_1D_2s_1s_2\notag\\
=&s_2s_1D_2s_1D_2D_1+s_2s_1D_2D_1D_2s_1+s_2D_1D_2s_1D_2s_1\notag\\
&\qquad +D_2s_1D_2s_1s_2D_1+D_2D_1s_2s_1D_2s_1+s_2s_1D_2s_1D_2s_1.
\end{align}
From \eqref{G-relation5} we get that
\begin{align}\label{G2-equality3}
s_1D_2s_2D_1s_2s_1s_2s_1D_2s_1=&s_1D_2s_2D_2D_1s_2s_1s_2+s_1D_2D_1D_2s_1s_2+s_1D_2s_1D_2D_1s_2\notag\\
&+s_1D_2s_1s_2D_1D_2+s_1D_2s_1s_2s_1D_2D_1s_1+s_1D_2s_1D_2s_1s_2\notag\\
&+s_1D_2s_1s_2s_1D_2+s_1D_2D_1s_2s_1s_2+s_1D_2s_1s_2D_1s_2;
\end{align}
From \eqref{G-relation13} we get that
\begin{align}\label{G2-equality4}
s_1D_2s_1s_2s_1s_2D_1s_2D_2s_1=&s_1D_1D_2s_1s_2s_1D_2s_1+D_2D_1s_2s_1D_2s_1+s_2D_1D_2s_1D_2s_1\notag\\
&+s_2s_1D_2D_1D_2s_1+s_2s_1s_2D_1D_2s_2D_2s_1+s_2D_1s_2s_1D_2s_1\notag\\
&+s_2s_1s_2D_1D_2s_1+D_2s_1s_2s_1D_2s_1+s_2s_1D_2s_1D_2s_1.
\end{align}
Note that $s_1D_2s_1s_2D_1s_2=s_2D_1s_2s_1D_2s_1,$ $s_1s_2s_1s_2D_1s_2=s_2D_1s_2s_1s_2s_1$ and $D_2s_2D_2=-D_2$. Thus, by \eqref{G2-equality3} and \eqref{G2-equality4} we get that
\begin{align}\label{G2-equality5}
&s_1D_2D_1D_2s_1s_2+s_1D_2s_1D_2D_1s_2+s_1D_2s_1s_2D_1D_2\notag\\
&\qquad +s_1D_2s_1s_2s_1D_2D_1s_1+s_1D_2s_1D_2s_1s_2+s_1D_2s_1s_2s_1D_2\notag\\
=&s_1D_1D_2s_1s_2s_1D_2s_1+D_2D_1s_2s_1D_2s_1+s_2D_1D_2s_1D_2s_1\notag\\
&\qquad +s_2s_1D_2D_1D_2s_1+D_2s_1s_2s_1D_2s_1+s_2s_1D_2s_1D_2s_1.
\end{align}
Thus, in order to show that \eqref{G2-equality2} holds, it suffices to show that
\begin{align}\label{G2-equality6}
&D_1s_2s_1D_2s_1D_2+D_1D_2s_1D_2s_1s_2+D_2s_1s_2s_1D_2s_1+s_1D_1D_2s_1s_2s_1D_2s_1\notag\\
&\qquad=s_2s_1D_2s_1D_2D_1+D_2s_1D_2s_1s_2D_1+s_1D_2s_1s_2s_1D_2+s_1D_2s_1s_2s_1D_2D_1s_1,
\end{align}
which is equivalent to
\begin{align}\label{G2-equality7}
&D_1s_2s_1D_2s_1D_2+D_1D_2s_1D_2s_1s_2-D_1s_1D_2s_1s_2s_1D_2s_1\notag\\
=&s_2s_1D_2s_1D_2D_1+D_2s_1D_2s_1s_2D_1-s_1D_2s_1s_2s_1D_2s_1D_1.
\end{align}
By \eqref{G-relation8} it is easy to see that the left-hand side and the right-hand side of \eqref{G2-equality7} are both equal to $-s_2s_1D_2s_1s_2D_1=-D_1s_2s_1D_2s_1s_2.$

\eqref{G-relation17} It can be proved similarly.

\eqref{G-relation18} Set $\widetilde{D}_1 :=D_1,$ $\widetilde{D}_2 :=s_1D_2s_1,$ $\widetilde{D}_3 :=s_1s_2D_1s_2s_1,$ $\widetilde{D}_4 :=s_2s_1D_2s_1s_2,$ $\widetilde{D}_5 :=s_2D_1s_2,$ and $\widetilde{D}_6 :=D_2.$ Thus, from \eqref{G-relation7} it is easy to see that
\begin{align}\label{G2-equality8}
\widetilde{D}_3\widetilde{D}_5+\widetilde{D}_1\widetilde{D}_3-\widetilde{D}_5\widetilde{D}_1-\widetilde{D}_3=0,
\end{align}
and
\begin{align}\label{G2-equality9}
\widetilde{D}_5\widetilde{D}_3+\widetilde{D}_3\widetilde{D}_1-\widetilde{D}_1\widetilde{D}_5-\widetilde{D}_3=0.
\end{align}
From \eqref{G2-equality8} we deduce that $\widetilde{D}_1\widetilde{D}_5\widetilde{D}_1=\widetilde{D}_1\widetilde{D}_3\widetilde{D}_5$; from \eqref{G2-equality9} we deduce that $\widetilde{D}_1\widetilde{D}_5\widetilde{D}_1=\widetilde{D}_5\widetilde{D}_3\widetilde{D}_1.$ From \eqref{G2-equality9} we also get that $\widetilde{D}_5\widetilde{D}_3\widetilde{D}_1=\widetilde{D}_5\widetilde{D}_1\widetilde{D}_5.$ So we have $\widetilde{D}_1\widetilde{D}_5\widetilde{D}_1=\widetilde{D}_5\widetilde{D}_1\widetilde{D}_5,$ which is exactly \eqref{G-relation18}.

\eqref{G-relation19} It can be proved similarly.

\eqref{G-relation20} By \eqref{G-relation9} we get that
\begin{align}
s_2D_2D_1&D_2s_1s_2D_1s_2+s_2D_2s_1D_2D_1s_2D_1s_2+s_2D_2s_1D_2s_1s_2D_1s_2\notag\\
&\qquad\qquad\qquad +s_2D_2s_1s_2D_1D_2D_1s_2+s_2D_2s_1s_2D_1s_2D_1s_2=\notag\\
&s_2D_1D_2D_1D_2s_1+s_2D_1D_2s_1D_2D_1+s_2D_1D_2s_1D_2s_1+s_2D_1D_2s_1s_2D_1D_2s_2\notag\\
&+s_2D_1D_2s_1s_2D_1+s_2s_1D_2D_1D_2D_1+s_2s_1D_2D_1D_2s_1+s_2s_1D_2D_1s_2D_1D_2s_2\notag\\
&+s_2s_1D_2D_1s_2D_1+s_2s_1D_2s_1D_2D_1+s_2s_1D_2s_1D_2s_1+s_2s_1D_2s_1s_2D_1D_2s_2\notag\\
&+s_2s_1D_2s_1s_2D_1+s_2s_1s_2D_1D_2D_1D_2s_2+s_2s_1s_2D_1D_2D_1+s_2s_1s_2D_1s_2D_1D_2s_2\notag\\
&+s_2s_1s_2D_1s_2D_1.\label{G2-equality10}
\end{align}
By the presentation we have
\begin{align}
s_2D_2D_1&D_2s_1s_2D_1s_2+s_2D_2s_1D_2D_1s_2D_1s_2+s_2D_2s_1D_2s_1s_2D_1s_2\notag\\
&\qquad\qquad +s_2D_2s_1s_2D_1D_2D_1s_2+s_2D_2s_1s_2D_1s_2D_1s_2=\notag\\
&(-D_2s_2+s_2-1)D_1D_2s_1s_2D_1s_2+(-D_2s_2+s_2-1)s_1D_2D_1s_2D_1s_2\notag\\
&\qquad+(-D_2s_2+s_2-1)s_1D_2s_1s_2D_1s_2+(-D_2s_2+s_2-1)s_1s_2D_1D_2D_1s_2\notag\\
&\qquad\qquad+(-D_2s_2+s_2-1)s_1s_2D_1s_2D_1s_2,\label{G2-equality11}
\end{align}

and

\begin{align}
s_2D_1&D_2s_1s_2D_1D_2s_2+s_2s_1D_2D_1s_2D_1D_2s_2+s_2s_1D_2s_1s_2D_1D_2s_2\notag\\
&\qquad\qquad +s_2s_1s_2D_1D_2D_1D_2s_2+s_2s_1s_2D_1s_2D_1D_2s_2=\notag\\
&s_2D_1D_2s_1s_2D_1(-s_2D_2+s_2-1)+s_2s_1D_2D_1s_2D_1(-s_2D_2+s_2-1)\notag\\
&\qquad+s_2s_1D_2s_1s_2D_1(-s_2D_2+s_2-1)+s_2s_1s_2D_1D_2D_1(-s_2D_2+s_2-1)\notag\\
&\qquad\qquad +s_2s_1s_2D_1s_2D_1(-s_2D_2+s_2-1).\label{G2-equality12}
\end{align}

By \eqref{G-relation15} we have
\begin{align}
(s_2D_1&D_2s_1s_2D_1+s_2s_1D_2D_1s_2D_1+s_2s_1D_2s_1s_2D_1\notag\\
&\qquad\quad+s_2s_1s_2D_1D_2D_1+s_2s_1s_2D_1s_2D_1)s_2D_2=\notag\\
&(D_1D_2D_1s_2s_1s_2+D_1s_2D_1D_2s_1s_2+D_1s_2D_1s_2s_1s_2\notag\\
&\qquad+D_1s_2s_1D_2D_1s_2+D_1s_2s_1s_2D_1D_2+D_1s_2s_1s_2D_1s_2\notag\\
&\qquad\qquad-D_2D_1s_2s_1s_2D_1-s_2D_1s_2s_1s_2D_1+s_2s_1D_2s_1s_2D_1)s_2D_2=\notag\\
&D_1D_2D_1s_2s_1D_2+D_1s_2D_1D_2s_1D_2+D_1s_2D_1s_2s_1D_2+D_1s_2s_1D_2D_1D_2\notag\\
&\qquad+D_1s_2s_1D_2s_1D_2-D_2D_1s_2s_1s_2D_1s_2D_2-s_2D_1s_2s_1s_2D_1s_2D_2.\label{G2-equality13}
\end{align}

By \eqref{G-relation15} we also have
\begin{align}
&s_2D_1D_2D_1s_2s_1+s_2D_1s_2D_1D_2s_1+s_2D_1s_2D_1s_2s_1\notag\\
&\qquad \qquad+s_2D_1s_2s_1D_2D_1+s_2D_1s_2s_1s_2D_1D_2s_2+s_2D_1s_2s_1s_2D_1\notag\\
=&s_1s_2D_1D_2D_1s_2+s_1D_2D_1s_2D_1s_2+s_1s_2D_1s_2D_1s_2\notag\\
&\qquad \qquad+D_1D_2s_1s_2D_1s_2+s_2D_2D_1s_2s_1s_2D_1s_2+D_1s_2s_1s_2D_1s_2.\label{G2-equality14}
\end{align}
Thus, we get that
\begin{align}
D_2s_2&(D_1D_2s_1s_2D_1s_2+s_1D_2D_1s_2D_1s_2+s_1D_2s_1s_2D_1s_2\notag\\
&\qquad\qquad +s_1s_2D_1D_2D_1s_2+s_1s_2D_1s_2D_1s_2)=\notag\\
&D_2s_2(s_2D_1D_2D_1s_2s_1+s_2D_1s_2D_1D_2s_1+s_2D_1s_2D_1s_2s_1\notag\\
&\qquad+s_2D_1s_2s_1D_2D_1+s_2D_1s_2s_1s_2D_1D_2s_2+s_2D_1s_2s_1s_2D_1\notag\\
&\qquad\qquad-s_2D_2D_1s_2s_1s_2D_1s_2-D_1s_2s_1s_2D_1s_2+s_1D_2s_1s_2D_1s_2)=\notag\\
&D_2D_1D_2D_1s_2s_1+D_2D_1s_2D_1D_2s_1+D_2D_1s_2D_1s_2s_1\notag\\
&\qquad +D_2D_1s_2s_1D_2D_1+D_2D_1s_2s_1s_2D_1D_2s_2+D_2D_1s_2s_1s_2D_1\notag\\
&\qquad\qquad-D_2D_1s_2s_1s_2D_1s_2-D_2s_2D_1s_2s_1s_2D_1s_2+D_2s_2s_1D_2s_1s_2D_1s_2.\label{G2-equality15}
\end{align}

So we have
\begin{align}
D_1D_2&D_1s_2s_1D_2+D_1s_2D_1D_2s_1D_2+D_1s_2D_1s_2s_1D_2\notag\\
&\qquad\qquad\qquad +D_1s_2s_1D_2D_1D_2+D_1s_2s_1D_2s_1D_2=\notag\\
&s_2D_1D_2D_1D_2s_1+s_2D_1D_2s_1D_2D_1+s_2D_1D_2s_1D_2s_1+s_2D_1D_2s_1s_2D_1\notag\\
&+s_2s_1D_2D_1D_2D_1+s_2s_1D_2D_1D_2s_1+s_2s_1D_2D_1s_2D_1+s_2s_1D_2s_1D_2D_1\notag\\
&+s_2s_1D_2s_1D_2s_1+s_2s_1D_2s_1s_2D_1+s_2s_1s_2D_1D_2D_1+s_2s_1s_2D_1s_2D_1\notag\\
&+D_1D_2s_1s_2D_1s_2+s_1D_2D_1s_2D_1s_2+s_1D_2s_1s_2D_1s_2+s_1s_2D_1D_2D_1s_2\notag\\
&+s_1s_2D_1s_2D_1s_2-s_2D_1D_2s_1s_2D_1-s_2s_1D_2D_1s_2D_1-s_2s_1D_2s_1s_2D_1\notag\\
&-s_2s_1s_2D_1D_2D_1-s_2s_1s_2D_1s_2D_1+D_2D_1D_2D_1s_2s_1+D_2D_1s_2D_1D_2s_1\notag\\
&+D_2D_1s_2D_1s_2s_1+D_2D_1s_2s_1D_2D_1+D_2D_1s_2s_1s_2D_1D_2s_2\notag\\
&+D_2D_1s_2s_1s_2D_1-D_2D_1s_2s_1s_2D_1s_2-D_2s_2D_1s_2s_1s_2D_1s_2\notag\\
&+D_2D_1s_2s_1D_2s_1+D_2D_1s_2s_1s_2D_1s_2D_2+s_2D_1s_2s_1s_2D_1s_2D_2.\label{G2-equality16}
\end{align}

Note that
\begin{align}
D_2D_1s_2s_1s_2D_1D_2s_2+D_2D_1s_2s_1s_2D_1-D_2D_1s_2s_1s_2D_1s_2+D_2D_1s_2s_1s_2D_1s_2D_2=0,\label{G2-equality17}
\end{align}
and
\begin{align}
-D_2s_2&D_1s_2s_1s_2D_1s_2+s_2D_1s_2s_1s_2D_1s_2D_2=\notag\\
&\qquad-D_2s_1s_2D_1s_2D_1-D_2D_1s_2D_1s_2s_1-D_2s_1s_2D_1s_2s_1\notag\\
&\qquad\qquad+s_1s_2D_1s_2D_1D_2+D_1s_2D_1s_2s_1D_2+s_1s_2D_1s_2s_1D_2=\notag\\
&-D_2s_1s_2D_1s_2D_1-D_2D_1s_2D_1s_2s_1+s_1s_2D_1s_2D_1D_2+D_1s_2D_1s_2s_1D_2.\label{G2-equality18}
\end{align}

By \eqref{G2-equality16}, in order to prove \eqref{G-relation20} it suffices to show that
\begin{align}
D_1D_2&s_1s_2D_1s_2+s_1D_2D_1s_2D_1s_2+s_1D_2s_1s_2D_1s_2+s_1s_2D_1D_2D_1s_2\notag\\
&\qquad+s_1s_2D_1s_2D_1s_2+s_1s_2D_1s_2D_1D_2+D_1s_2D_1s_2s_1D_2=\notag\\
s_2D_1&D_2D_1s_2s_1+s_2D_1s_2D_1D_2s_1+s_2D_1s_2D_1s_2s_1+s_2D_1s_2s_1D_2D_1\notag\\
&\qquad+s_2D_1s_2s_1D_2s_1+D_2s_1s_2D_1s_2D_1+D_2D_1s_2D_1s_2s_1,\label{G2-equality19}
\end{align}
which can be proved in exactly the same way as \eqref{G2-equality2} by noting that $s_1D_2s_1s_2D_1s_2=s_2D_1s_2s_1D_2s_1.$

\eqref{G-relation21} By \eqref{G-relation9} and \eqref{G-relation20} we get that
\begin{align}
D_1D_2D_1D_2&s_1s_2+D_1D_2s_1D_2D_1s_2+D_1D_2s_1D_2s_1s_2+D_1D_2s_1s_2D_1D_2\notag\\
+D_1D_2s_1&s_2D_1s_2+s_1D_2D_1D_2D_1s_2+s_1D_2D_1D_2s_1s_2+s_1D_2D_1s_2D_1D_2\notag\\
+s_1D_2&D_1s_2D_1s_2+s_1D_2s_1D_2D_1s_2+s_1D_2s_1D_2s_1s_2+s_1D_2s_1s_2D_1D_2\notag\\
+s_1&D_2s_1s_2D_1s_2+s_1s_2D_1D_2D_1D_2+s_1s_2D_1D_2D_1s_2+s_1s_2D_1s_2D_1D_2\notag\\
+&s_1s_2D_1s_2D_1s_2+D_1D_2D_1s_2s_1D_2+D_1s_2D_1D_2s_1D_2+D_1s_2D_1s_2s_1D_2\notag\\
&+D_1s_2s_1D_2D_1D_2+D_1s_2s_1D_2s_1D_2=\notag\\
D_2D_1D_2s_1&s_2D_1+D_2s_1D_2D_1s_2D_1+D_2s_1D_2s_1s_2D_1+D_2s_1s_2D_1D_2D_1\notag\\
+D_2s_1s_2&D_1s_2D_1+D_2D_1D_2D_1s_2s_1+D_2D_1s_2D_1D_2s_1+D_2D_1s_2D_1s_2s_1\notag\\
+D_2D_1&s_2s_1D_2D_1+D_2D_1s_2s_1D_2s_1+s_2D_1D_2D_1D_2s_1+s_2D_1D_2D_1s_2s_1\notag\\
+s_2&D_1D_2s_1D_2D_1+s_2D_1D_2s_1D_2s_1+s_2D_1s_2D_1D_2s_1+s_2D_1s_2D_1s_2s_1\notag\\
+&s_2D_1s_2s_1D_2D_1+s_2D_1s_2s_1D_2s_1+s_2s_1D_2D_1D_2D_1+s_2s_1D_2D_1D_2s_1\notag\\
&+s_2s_1D_2s_1D_2D_1+s_2s_1D_2s_1D_2s_1.\label{G2-equality20}
\end{align}

By \eqref{G2-equality20}, we see that, in order to prove \eqref{G-relation21}, it suffices to show that
\begin{align}
D_1D_2s_1&D_2s_1s_2+D_1D_2s_1s_2D_1s_2+s_1D_2D_1D_2s_1s_2+s_1D_2D_1s_2D_1s_2\notag\\
+s_1D_2&s_1D_2D_1s_2+s_1D_2s_1D_2s_1s_2+s_1D_2s_1s_2D_1D_2+s_1D_2s_1s_2D_1s_2\notag\\
+s_1&s_2D_1D_2D_1s_2+s_1s_2D_1s_2D_1D_2+s_1s_2D_1s_2D_1s_2\notag\\
+&D_1s_2D_1s_2s_1D_2+D_1s_2s_1D_2s_1D_2=\notag\\
D_2s_1D_2&s_1s_2D_1+D_2s_1s_2D_1s_2D_1+D_2D_1s_2D_1s_2s_1+D_2D_1s_2s_1D_2s_1\notag\\
+s_2D_1&D_2D_1s_2s_1+s_2D_1D_2s_1D_2s_1+s_2D_1s_2D_1D_2s_1+s_2D_1s_2D_1s_2s_1\notag\\
+s_2&D_1s_2s_1D_2D_1+s_2D_1s_2s_1D_2s_1+s_2s_1D_2D_1D_2s_1\notag\\
+&s_2s_1D_2s_1D_2D_1+s_2s_1D_2s_1D_2s_1.\label{G2-equality21}
\end{align}

By \eqref{G2-equality2}, we see that, in order to show \eqref{G2-equality21}, it suffices to prove that
\begin{align}
D_1D_2&s_1s_2D_1s_2+s_1D_2D_1s_2D_1s_2+s_1D_2s_1s_2D_1s_2+s_1s_2D_1D_2D_1s_2\notag\\
&\qquad +s_1s_2D_1s_2D_1D_2+s_1s_2D_1s_2D_1s_2+D_1s_2D_1s_2s_1D_2=\notag\\
D_2s_1&s_2D_1s_2D_1+D_2D_1s_2D_1s_2s_1+s_2D_1D_2D_1s_2s_1+s_2D_1s_2D_1D_2s_1\notag\\
&+s_2D_1s_2D_1s_2s_1+s_2D_1s_2s_1D_2D_1+s_2D_1s_2s_1D_2s_1,\label{G2-equality22}
\end{align}
which is exactly the equality \eqref{G2-equality19}.

\eqref{G-relation22} By \eqref{G-relation9} we get that
\begin{align}
D_2&D_1D_2D_1D_2s_1+D_2D_1D_2s_1D_2D_1+D_2D_1D_2s_1D_2s_1+D_2D_1D_2s_1s_2D_1D_2s_2\notag\\
&+D_2D_1D_2s_1s_2D_1+D_2s_1D_2D_1D_2D_1+D_2s_1D_2D_1D_2s_1+D_2s_1D_2D_1s_2D_1D_2s_2\notag\\
&+D_2s_1D_2D_1s_2D_1+D_2s_1D_2s_1D_2D_1+D_2s_1D_2s_1D_2s_1+D_2s_1D_2s_1s_2D_1D_2s_2\notag\\
&+D_2s_1D_2s_1s_2D_1+D_2s_1s_2D_1D_2D_1D_2s_2+D_2s_1s_2D_1D_2D_1\notag\\
&+D_2s_1s_2D_1s_2D_1D_2s_2+D_2s_1s_2D_1s_2D_1=\notag\\
D_1&D_2D_1D_2s_1+D_1D_2s_1D_2D_1+D_1D_2s_1D_2s_1+D_1D_2s_1s_2D_1D_2s_2\notag\\
&+D_1D_2s_1s_2D_1+s_1D_2D_1D_2D_1+s_1D_2D_1D_2s_1+s_1D_2D_1s_2D_1D_2s_2\notag\\
&+s_1D_2D_1s_2D_1+s_1D_2s_1D_2D_1+s_1D_2s_1D_2s_1+s_1D_2s_1s_2D_1D_2s_2\notag\\
&+s_1D_2s_1s_2D_1+s_1s_2D_1D_2D_1D_2s_2+s_1s_2D_1D_2D_1\notag\\
&+s_1s_2D_1s_2D_1D_2s_2+s_1s_2D_1s_2D_1.\label{G2-equality23}
\end{align}

Similarly, by \eqref{G-relation20} we get that
\begin{align}
s_2&D_2D_1D_2D_1s_2s_1D_2+s_2D_2D_1s_2D_1D_2s_1D_2+s_2D_2D_1s_2D_1s_2s_1D_2+s_2D_2D_1s_2s_1D_2D_1D_2\notag\\
&+s_2D_2D_1s_2s_1D_2s_1D_2+D_1D_2D_1D_2s_1D_2+D_1D_2D_1s_2s_1D_2+D_1D_2s_1D_2D_1D_2\notag\\
&+D_1D_2s_1D_2s_1D_2+D_1s_2D_1D_2s_1D_2+D_1s_2D_1s_2s_1D_2+D_1s_2s_1D_2D_1D_2\notag\\
&+D_1s_2s_1D_2s_1D_2+s_1D_2D_1D_2D_1D_2+s_1D_2D_1D_2s_1D_2\notag\\
&+s_1D_2s_1D_2D_1D_2+s_1D_2s_1D_2s_1D_2=\notag\\
s_2&D_2D_1D_2D_1s_2s_1+s_2D_2D_1s_2D_1D_2s_1+s_2D_2D_1s_2D_1s_2s_1+s_2D_2D_1s_2s_1D_2D_1\notag\\
&+s_2D_2D_1s_2s_1D_2s_1+D_1D_2D_1D_2s_1+D_1D_2D_1s_2s_1+D_1D_2s_1D_2D_1\notag\\
&+D_1D_2s_1D_2s_1+D_1s_2D_1D_2s_1+D_1s_2D_1s_2s_1+D_1s_2s_1D_2D_1\notag\\
&+D_1s_2s_1D_2s_1+s_1D_2D_1D_2D_1+s_1D_2D_1D_2s_1\notag\\
&+s_1D_2s_1D_2D_1+s_1D_2s_1D_2s_1.\label{G2-equality24}
\end{align}

By \eqref{G-relation10} and \eqref{G-relation19}, we see that, in order to prove \eqref{G-relation22}, it suffices to show that
\begin{align}
&D_2D_1D_2s_1s_2D_1D_2s_2+D_2D_1D_2s_1s_2D_1+D_2s_1D_2D_1s_2D_1D_2s_2+D_2s_1D_2D_1s_2D_1\notag\\
&+D_2s_1D_2s_1s_2D_1D_2s_2+D_2s_1D_2s_1s_2D_1+D_2s_1s_2D_1D_2D_1D_2s_2+D_2s_1s_2D_1D_2D_1\notag\\
&+D_2s_1s_2D_1s_2D_1D_2s_2+D_2s_1s_2D_1s_2D_1\notag\\
&+s_2D_2D_1s_2s_1D_2D_1+D_1s_2s_1D_2D_1+s_2D_2D_1s_2D_1D_2s_1+D_1s_2D_1D_2s_1\notag\\
&+s_2D_2D_1s_2s_1D_2s_1+D_1s_2s_1D_2s_1+s_2D_2D_1D_2D_1s_2s_1+D_1D_2D_1s_2s_1\notag\\
&+s_2D_2D_1s_2D_1s_2s_1+D_1s_2D_1s_2s_1=\notag\\
&D_1D_2s_1s_2D_1D_2s_2+D_1D_2s_1s_2D_1+s_1D_2D_1s_2D_1D_2s_2+s_1D_2D_1s_2D_1\notag\\
&+s_1D_2s_1s_2D_1D_2s_2+s_1D_2s_1s_2D_1+s_1s_2D_1D_2D_1D_2s_2+s_1s_2D_1D_2D_1\notag\\
&+s_1s_2D_1s_2D_1D_2s_2+s_1s_2D_1s_2D_1\notag\\
&+s_2D_2D_1s_2s_1D_2D_1D_2+D_1s_2s_1D_2D_1D_2+s_2D_2D_1s_2D_1D_2s_1D_2+D_1s_2D_1D_2s_1D_2\notag\\
&+s_2D_2D_1s_2s_1D_2s_1D_2+D_1s_2s_1D_2s_1D_2+s_2D_2D_1D_2D_1s_2s_1D_2+D_1D_2D_1s_2s_1D_2\notag\\
&+s_2D_2D_1s_2D_1s_2s_1D_2+D_1s_2D_1s_2s_1D_2,\label{G2-equality25}
\end{align}
which is equivalent to
\begin{align}
&D_2D_1D_2s_1s_2D_1(-s_2D_2+s_2)+D_2s_1D_2D_1s_2D_1(-s_2D_2+s_2)\notag\\
&+D_2s_1D_2s_1s_2D_1(-s_2D_2+s_2)+D_2s_1s_2D_1D_2D_1(-s_2D_2+s_2)\notag\\
&+D_2s_1s_2D_1s_2D_1(-s_2D_2+s_2)\notag\\
&+(-D_2s_2+s_2)D_1s_2s_1D_2D_1+(-D_2s_2+s_2)D_1s_2D_1D_2s_1\notag\\
&+(-D_2s_2+s_2)D_1s_2s_1D_2s_1+(-D_2s_2+s_2)D_1D_2D_1s_2s_1\notag\\
&+(-D_2s_2+s_2)D_1s_2D_1s_2s_1=\notag\\
&D_1D_2s_1s_2D_1(-s_2D_2+s_2)+s_1D_2D_1s_2D_1(-s_2D_2+s_2)\notag\\
&+s_1D_2s_1s_2D_1(-s_2D_2+s_2)+s_1s_2D_1D_2D_1(-s_2D_2+s_2)\notag\\
&+s_1s_2D_1s_2D_1(-s_2D_2+s_2)\notag\\
&+(-D_2s_2+s_2)D_1s_2s_1D_2D_1D_2+(-D_2s_2+s_2)D_1s_2D_1D_2s_1D_2\notag\\
&+(-D_2s_2+s_2)D_1s_2s_1D_2s_1D_2+(-D_2s_2+s_2)D_1D_2D_1s_2s_1D_2\notag\\
&+(-D_2s_2+s_2)D_1s_2D_1s_2s_1D_2,\label{G2-equality26}
\end{align}

By \eqref{G2-equality19} we have
\begin{align}
&s_2D_1s_2s_1D_2D_1+s_2D_1s_2D_1D_2s_1+s_2D_1D_2D_1s_2s_1\notag\\
&\qquad+D_2s_1s_2D_1s_2D_1+D_2D_1s_2D_1s_2s_1+s_2D_1s_2D_1s_2s_1=\notag\\
&s_1s_2D_1s_2D_1D_2+s_1s_2D_1D_2D_1s_2+s_1D_2D_1s_2D_1s_2\notag\\
&\qquad +D_1s_2D_1s_2s_1D_2+D_1D_2s_1s_2D_1s_2+s_1s_2D_1s_2D_1s_2.\label{G2-equality27}
\end{align}

By \eqref{G2-equality27}, we see that \eqref{G2-equality26} is equivalent to
\begin{align}
&-D_2\big(s_2D_1s_2s_1D_2D_1+s_2D_1s_2D_1D_2s_1+s_2D_1D_2D_1s_2s_1+D_2s_1s_2D_1s_2D_1\notag\\
&+D_2D_1s_2D_1s_2s_1+s_2D_1s_2D_1s_2s_1-s_1s_2D_1s_2D_1D_2-D_1s_2D_1s_2s_1D_2\big)D_2\notag\\
&+D_2\big(s_2D_1s_2s_1D_2D_1+s_2D_1s_2D_1D_2s_1+s_2D_1D_2D_1s_2s_1\notag\\
&+D_2s_1s_2D_1s_2D_1+D_2D_1s_2D_1s_2s_1+s_2D_1s_2D_1s_2s_1\notag\\
&-s_1s_2D_1s_2D_1D_2-D_1s_2D_1s_2s_1D_2\big)\notag\\
&+D_2s_1D_2s_1s_2D_1(-s_2D_2+s_2)+D_2\big(-s_2D_1s_2s_1D_2D_1-s_2D_1s_2D_1D_2s_1\notag\\
&-s_2D_1s_2s_1D_2s_1-s_2D_1D_2D_1s_2s_1-s_2D_1s_2D_1s_2s_1\big)\notag\\
&+s_2D_1s_2s_1D_2D_1+s_2D_1s_2D_1D_2s_1+s_2D_1s_2s_1D_2s_1\notag\\
&+s_2D_1D_2D_1s_2s_1+s_2D_1s_2D_1s_2s_1=\notag\\
&-D_2\big(s_2D_1s_2s_1D_2D_1+s_2D_1s_2D_1D_2s_1+s_2D_1s_2s_1D_2s_1+s_2D_1D_2D_1s_2s_1\notag\\
&+s_2D_1s_2D_1s_2s_1\big)D_2+\big(s_2D_1s_2s_1D_2D_1+s_2D_1s_2D_1D_2s_1+s_2D_1s_2s_1D_2s_1\notag\\
&+s_2D_1D_2D_1s_2s_1+s_2D_1s_2D_1s_2s_1-D_1D_2s_1s_2D_1s_2-s_1D_2D_1s_2D_1s_2\notag\\
&-s_1D_2s_1s_2D_1s_2-s_1s_2D_1D_2D_1s_2-s_1s_2D_1s_2D_1s_2\big)D_2\notag\\
&+D_1D_2s_1s_2D_1s_2+s_1D_2D_1s_2D_1s_2+s_1D_2s_1s_2D_1s_2\notag\\
&+s_1s_2D_1D_2D_1s_2+s_1s_2D_1s_2D_1s_2.\label{G2-equality28}
\end{align}

By noting that $s_2D_1s_2s_1D_2s_1=s_1D_2s_1s_2D_1s_2,$ we see that \eqref{G2-equality28} is equivalent to
\begin{align}
&-D_2\big(D_2s_1s_2D_1s_2D_1+D_2D_1s_2D_1s_2s_1-s_1s_2D_1s_2D_1D_2-D_1s_2D_1s_2s_1D_2\big)D_2\notag\\
&+D_2s_1D_2s_1s_2D_1(-s_2D_2+s_2)+D_2\big(D_2s_1s_2D_1s_2D_1+D_2D_1s_2D_1s_2s_1\notag\\
&-s_1s_2D_1s_2D_1D_2-D_1s_2D_1s_2s_1D_2-s_2D_1s_2s_1D_2s_1\big)\notag\\
&+s_1s_2D_1s_2D_1D_2+D_1s_2D_1s_2s_1D_2-D_2s_1s_2D_1s_2D_1-D_2D_1s_2D_1s_2s_1=\notag\\
&-D_2s_2D_1s_2s_1D_2s_1D_2+s_1s_2D_1s_2D_1D_2+D_1s_2D_1s_2s_1D_2\notag\\
&-D_2s_1s_2D_1s_2D_1D_2-D_2D_1s_2D_1s_2s_1D_2,\label{G2-equality29}
\end{align}
which is equivalent to
\begin{align}
-D_2s_1D_2s_1s_2D_1s_2D_2+D_2s_1D_2s_1s_2D_1s_2-D_2s_2D_1s_2s_1D_2s_1=-D_2s_2D_1s_2s_1D_2s_1D_2.\label{G2-equality30}
\end{align}
By $s_2D_1s_2s_1D_2s_1=s_1D_2s_1s_2D_1s_2,$ we see that \eqref{G2-equality30} is true.

\eqref{G-relation23} It can be proved similarly.
\end{proof}

Recall that $\widetilde{D}_1 :=D_1,$ $\widetilde{D}_2 :=s_1D_2s_1,$ $\widetilde{D}_3 :=s_1s_2D_1s_2s_1,$ $\widetilde{D}_4 :=s_2s_1D_2s_1s_2,$ $\widetilde{D}_5 :=s_2D_1s_2,$ and $\widetilde{D}_6 :=D_2.$

We have that \eqref{G-relation5} corresponds to the element
\begin{align*}
\widetilde{D}_6\widetilde{D}_1-\widetilde{D}_1\widetilde{D}_2-\widetilde{D}_2\widetilde{D}_3-\widetilde{D}_3\widetilde{D}_4-
\widetilde{D}_4\widetilde{D}_5-\widetilde{D}_5\widetilde{D}_6+\widetilde{D}_2+\widetilde{D}_3+\widetilde{D}_4+\widetilde{D}_5
\end{align*}
in $K_{12}(W_2);$

\eqref{G-relation6} corresponds to the element
$\widetilde{D}_3\widetilde{D}_6-\widetilde{D}_6\widetilde{D}_3$ in $K_{12}(W_2);$

\eqref{G-relation7} corresponds to the element
$\widetilde{D}_1\widetilde{D}_3+\widetilde{D}_3\widetilde{D}_5-\widetilde{D}_5\widetilde{D}_1-\widetilde{D}_3$ in $K_{12}(W_2);$

\eqref{G-relation8} corresponds to the element
$\widetilde{D}_4\widetilde{D}_2+\widetilde{D}_6\widetilde{D}_4-\widetilde{D}_2\widetilde{D}_6-\widetilde{D}_4$ in $K_{12}(W_2);$

\eqref{G-relation9} corresponds to the element
\begin{align}\label{G2-K12W2-element1}
&\widetilde{D}_6\widetilde{D}_1\widetilde{D}_6\widetilde{D}_3-\widetilde{D}_6\widetilde{D}_2\widetilde{D}_1\widetilde{D}_3+
\widetilde{D}_6\widetilde{D}_2\widetilde{D}_3-\widetilde{D}_6\widetilde{D}_3\widetilde{D}_2\widetilde{D}_3+\widetilde{D}_6\widetilde{D}_3\notag\\
&-\widetilde{D}_1\widetilde{D}_6\widetilde{D}_1\widetilde{D}_6+\widetilde{D}_1\widetilde{D}_6\widetilde{D}_2\widetilde{D}_1-
\widetilde{D}_1\widetilde{D}_6\widetilde{D}_2+\widetilde{D}_1\widetilde{D}_6\widetilde{D}_3\widetilde{D}_2-\widetilde{D}_1\widetilde{D}_6\widetilde{D}_3\notag\\
&-\widetilde{D}_2\widetilde{D}_1\widetilde{D}_2\widetilde{D}_1+\widetilde{D}_2\widetilde{D}_1+
\widetilde{D}_2\widetilde{D}_1\widetilde{D}_2-\widetilde{D}_2-\widetilde{D}_2\widetilde{D}_1\widetilde{D}_3\widetilde{D}_2\notag\\
&+\widetilde{D}_2\widetilde{D}_3\widetilde{D}_2+\widetilde{D}_2\widetilde{D}_1\widetilde{D}_3-\widetilde{D}_2\widetilde{D}_3
-\widetilde{D}_3\widetilde{D}_2\widetilde{D}_3\widetilde{D}_2+\widetilde{D}_3\widetilde{D}_2+\widetilde{D}_3\widetilde{D}_2\widetilde{D}_3-\widetilde{D}_3
\end{align}
in $K_{12}(W_2);$

\eqref{G-relation10} corresponds to the element
\begin{align}\label{G2-K12W2-element2}
&\widetilde{D}_2\widetilde{D}_6\widetilde{D}_1\widetilde{D}_6-\widetilde{D}_2\widetilde{D}_1\widetilde{D}_2\widetilde{D}_6+
\widetilde{D}_2\widetilde{D}_6+\widetilde{D}_1\widetilde{D}_6\widetilde{D}_2\widetilde{D}_6\notag\\
&\qquad-\widetilde{D}_6\widetilde{D}_2\widetilde{D}_6\widetilde{D}_1+\widetilde{D}_6\widetilde{D}_2\widetilde{D}_1\widetilde{D}_2-
\widetilde{D}_6\widetilde{D}_2-\widetilde{D}_6\widetilde{D}_1\widetilde{D}_6\widetilde{D}_2
\end{align}
in $K_{12}(W_2);$

\eqref{G-relation11} corresponds to the element
\begin{align}\label{G2-K12W2-element3}
&\widetilde{D}_5\widetilde{D}_1\widetilde{D}_6\widetilde{D}_1-\widetilde{D}_5\widetilde{D}_6\widetilde{D}_5\widetilde{D}_1+
\widetilde{D}_5\widetilde{D}_1+\widetilde{D}_6\widetilde{D}_1\widetilde{D}_5\widetilde{D}_1\notag\\
&\qquad-\widetilde{D}_1\widetilde{D}_5\widetilde{D}_1\widetilde{D}_6+\widetilde{D}_1\widetilde{D}_5\widetilde{D}_6\widetilde{D}_5-
\widetilde{D}_1\widetilde{D}_5-\widetilde{D}_1\widetilde{D}_6\widetilde{D}_1\widetilde{D}_5
\end{align}
in $K_{12}(W_2);$

\eqref{G-relation12} corresponds to the element
\begin{align*}
\widetilde{D}_1\widetilde{D}_2\widetilde{D}_3\widetilde{D}_4\widetilde{D}_5\widetilde{D}_6-
\widetilde{D}_6\widetilde{D}_5\widetilde{D}_4\widetilde{D}_3\widetilde{D}_2\widetilde{D}_1
\end{align*}
in $K_{12}(W_2).$

Set $\alpha_{1}^{\vee} :=\varepsilon_{1}^{\vee}-\varepsilon_{2}^{\vee},$ $\alpha_{2}^{\vee} :=2\varepsilon_{2}^{\vee}-\varepsilon_{1}^{\vee}-\varepsilon_{3}^{\vee}$ and also $\alpha_1 :=\varepsilon_1-\varepsilon_2,$ $\alpha_2 :=\frac{2}{3}\varepsilon_2-\frac{1}{3}\varepsilon_1-\frac{1}{3}\varepsilon_3,$ where $\varepsilon_1,$ $\varepsilon_2$ and $\varepsilon_3$ are a set of bases of the Euclidean space $\mathbb{R}^{3}$ and $\varepsilon_1^{\vee},$ $\varepsilon_2^{\vee}$ and $\varepsilon_3^{\vee}$ are the dual bases. We have
\begin{align*}
a_{11} :=\langle\alpha_{1}^{\vee},\alpha_{1}\rangle=2,\quad a_{22} :=\langle\alpha_{2}^{\vee},\alpha_{2}\rangle=2,\quad a_{12} :=\langle\alpha_{2}^{\vee},\alpha_{1}\rangle=-3,\quad a_{21} :=\langle\alpha_{1}^{\vee},\alpha_{2}\rangle=-1.
\end{align*}
Let $\mathcal{Q}_2$ be the field of fractions of the Laurent polynomial ring $\mathcal{L}_2=\mathbb{Z}[t_1^{\pm 1}, t_2^{\pm 1}].$ We define the action of $W_2$ on $\mathcal{Q}_2$ by
\begin{align*}
s_1(t_1)=t_1^{-1},\quad s_1(t_2)=t_1^{-a_{12}}t_2=t_1^{3}t_2,\quad s_2(t_1)=t_2^{-a_{21}}t_1=t_1t_2,\quad s_2(t_2)=t_2^{-1}.
\end{align*}

Recall from [BK, Proposition 7.42] that the assignments $D_{i} \mapsto \frac{1}{1-t_{i}}(1-s_{i})$ and $s_{i}\mapsto s_{i}$ $(1\leq i\leq2),$ define a homomorphism of algebras $\hat{p}_{W_2} :\hat{\bf{H}}(W_{2})\rightarrow \mathcal{Q}_2\rtimes \mathbb{Z}W_{2}.$ Then we have the following lemma.
\begin{lemma}\label{G2-lemma1}
We have

$(1)$ $\hat{p}_{W_2}(\widetilde{D}_6\widetilde{D}_1-\widetilde{D}_1\widetilde{D}_2-\widetilde{D}_2\widetilde{D}_3-\widetilde{D}_3\widetilde{D}_4-
\widetilde{D}_4\widetilde{D}_5-\widetilde{D}_5\widetilde{D}_6+\widetilde{D}_2+\widetilde{D}_3+\widetilde{D}_4+\widetilde{D}_5)=0.$

$(2)$ $\hat{p}_{W_2}(\widetilde{D}_3\widetilde{D}_6-\widetilde{D}_6\widetilde{D}_3)=0.$

$(3)$ $\hat{p}_{W_2}(\widetilde{D}_1\widetilde{D}_3+\widetilde{D}_3\widetilde{D}_5-\widetilde{D}_5\widetilde{D}_1-\widetilde{D}_3)=0.$

$(4)$ $\hat{p}_{W_2}(\widetilde{D}_4\widetilde{D}_2+\widetilde{D}_6\widetilde{D}_4-\widetilde{D}_2\widetilde{D}_6-\widetilde{D}_4)=0.$

$(5)$
\begin{align*}
\hat{p}_{W_2}&(\widetilde{D}_6\widetilde{D}_1\widetilde{D}_6\widetilde{D}_3-\widetilde{D}_6\widetilde{D}_2\widetilde{D}_1\widetilde{D}_3+
\widetilde{D}_6\widetilde{D}_2\widetilde{D}_3-\widetilde{D}_6\widetilde{D}_3\widetilde{D}_2\widetilde{D}_3+\widetilde{D}_6\widetilde{D}_3\\
&-\widetilde{D}_1\widetilde{D}_6\widetilde{D}_1\widetilde{D}_6+\widetilde{D}_1\widetilde{D}_6\widetilde{D}_2\widetilde{D}_1-
\widetilde{D}_1\widetilde{D}_6\widetilde{D}_2+\widetilde{D}_1\widetilde{D}_6\widetilde{D}_3\widetilde{D}_2-\widetilde{D}_1\widetilde{D}_6\widetilde{D}_3\\&
-\widetilde{D}_2\widetilde{D}_1\widetilde{D}_2\widetilde{D}_1+\widetilde{D}_2\widetilde{D}_1+
\widetilde{D}_2\widetilde{D}_1\widetilde{D}_2-\widetilde{D}_2-\widetilde{D}_2\widetilde{D}_1\widetilde{D}_3\widetilde{D}_2\\
&+\widetilde{D}_2\widetilde{D}_3\widetilde{D}_2+\widetilde{D}_2\widetilde{D}_1\widetilde{D}_3-\widetilde{D}_2\widetilde{D}_3
-\widetilde{D}_3\widetilde{D}_2\widetilde{D}_3\widetilde{D}_2+\widetilde{D}_3\widetilde{D}_2+\widetilde{D}_3\widetilde{D}_2\widetilde{D}_3-\widetilde{D}_3)
=0.
\end{align*}

$(6)$
\begin{align*}
\hat{p}_{W_2}(&\widetilde{D}_2\widetilde{D}_6\widetilde{D}_1\widetilde{D}_6-\widetilde{D}_2\widetilde{D}_1\widetilde{D}_2\widetilde{D}_6+
\widetilde{D}_2\widetilde{D}_6+\widetilde{D}_1\widetilde{D}_6\widetilde{D}_2\widetilde{D}_6\\
&\qquad-\widetilde{D}_6\widetilde{D}_2\widetilde{D}_6\widetilde{D}_1+\widetilde{D}_6\widetilde{D}_2\widetilde{D}_1\widetilde{D}_2-
\widetilde{D}_6\widetilde{D}_2-\widetilde{D}_6\widetilde{D}_1\widetilde{D}_6\widetilde{D}_2)=0.
\end{align*}

$(7)$ \begin{align*}
\hat{p}_{W_2}(&\widetilde{D}_5\widetilde{D}_1\widetilde{D}_6\widetilde{D}_1-\widetilde{D}_5\widetilde{D}_6\widetilde{D}_5\widetilde{D}_1+
\widetilde{D}_5\widetilde{D}_1+\widetilde{D}_6\widetilde{D}_1\widetilde{D}_5\widetilde{D}_1\\
&\qquad-\widetilde{D}_1\widetilde{D}_5\widetilde{D}_1\widetilde{D}_6+\widetilde{D}_1\widetilde{D}_5\widetilde{D}_6\widetilde{D}_5-
\widetilde{D}_1\widetilde{D}_5-\widetilde{D}_1\widetilde{D}_6\widetilde{D}_1\widetilde{D}_5)=0.
\end{align*}

$(8)$ $\hat{p}_{W_2}(\widetilde{D}_1\widetilde{D}_2\widetilde{D}_3\widetilde{D}_4\widetilde{D}_5\widetilde{D}_6-
\widetilde{D}_6\widetilde{D}_5\widetilde{D}_4\widetilde{D}_3\widetilde{D}_2\widetilde{D}_1)=0.$
\end{lemma}
\begin{proof}
$(1)$ Set $\tau_1 :=\frac{1}{1-t_1}$ and $\tau_2 :=\frac{1}{1-t_2}$ such that $\hat{p}_{W_2}(\widetilde{D}_1)=\tau_1(1-s_1)$ and $\hat{p}_{W_2}(\widetilde{D}_6)=\tau_2(1-s_2).$ Therefore, we have
\begin{align*}
\hat{p}_{W_2}(\widetilde{D}_2)=s_1\tau_2(1-s_2)s_1=\tau_{12}(1-s_{12}),
\end{align*}
\begin{align*}
\hat{p}_{W_2}(\widetilde{D}_3)=s_1s_2\tau_1(1-s_1)s_2s_1=\tau_{121}(1-s_{121}),
\end{align*}
\begin{align*}
\hat{p}_{W_2}(\widetilde{D}_4)=s_2s_1\tau_2(1-s_2)s_1s_2=\tau_{212}(1-s_{212}),
\end{align*}
\begin{align*}
\hat{p}_{W_2}(\widetilde{D}_5)=s_2\tau_1(1-s_1)s_2=\tau_{21}(1-s_{21}),
\end{align*}
where $\tau_{12}=s_1\tau_{2}s_1=\frac{1}{1-t_1^{3}t_2},$ $\tau_{21}=s_2\tau_{1}s_2=\frac{1}{1-t_1t_2},$ $\tau_{121}=s_1s_2\tau_{1}s_2s_1=\frac{1}{1-t_1^{2}t_2},$ $\tau_{212}=s_2s_1\tau_{2}s_1s_2=\frac{1}{1-t_1^{3}t_2^{2}},$ $s_{12}=s_1s_2s_1,$ $s_{121}=s_1s_2s_1s_2s_1,$ $s_{212}=s_2s_1s_2s_1s_2,$ and $s_{21}=s_2s_1s_2.$

Thus, we get that
\begin{align*}
\hat{p}_{W_2}(\widetilde{D}_6\widetilde{D}_1)=\tau_2(1-s_2)\tau_1(1-s_1)=\tau_2\tau_1(1-s_1)-\tau_2\tau_{21}s_2(1-s_1);
\end{align*}
\begin{align*}
\hat{p}_{W_2}(\widetilde{D}_1\widetilde{D}_2)=\tau_1(1-s_1)\tau_{12}(1-s_{12})=\tau_1\tau_{12}(1-s_{12})-\tau_1\tau_{2}s_1(1-s_{12});
\end{align*}
\begin{align*}
\hat{p}_{W_2}(\widetilde{D}_2\widetilde{D}_3)=\tau_{12}(1-s_{12})\tau_{121}(1-s_{121})=\tau_{12}\tau_{121}(1-s_{121})-\tau_{12}(1-\tau_{1})s_{12}(1-s_{121});
\end{align*}
\begin{align*}
\hat{p}_{W_2}(\widetilde{D}_3\widetilde{D}_4)=\tau_{121}(1-s_{121})\tau_{212}(1-s_{212})=\tau_{121}\tau_{212}(1-s_{212})-\tau_{121}(1-\tau_{12})s_{121}(1-s_{212});
\end{align*}
\begin{align*}
\hat{p}_{W_2}(\widetilde{D}_4\widetilde{D}_5)=\tau_{212}(1-s_{212})\tau_{21}(1-s_{21})=\tau_{212}\tau_{21}(1-s_{21})-\tau_{212}(1-\tau_{121})s_{212}(1-s_{21});
\end{align*}
\begin{align*}
\hat{p}_{W_2}(\widetilde{D}_5\widetilde{D}_6)=\tau_{21}(1-s_{21})\tau_{2}(1-s_{2})=\tau_{21}\tau_{2}(1-s_{2})-\tau_{21}(1-\tau_{212})s_{21}(1-s_{2});
\end{align*}
Therefore, we have
\begin{align}
&\hat{p}_{W_2}(\widetilde{D}_6\widetilde{D}_1-\widetilde{D}_1\widetilde{D}_2-\widetilde{D}_2\widetilde{D}_3-\widetilde{D}_3\widetilde{D}_4-
\widetilde{D}_4\widetilde{D}_5-\widetilde{D}_5\widetilde{D}_6+\widetilde{D}_2+\widetilde{D}_3+\widetilde{D}_4+\widetilde{D}_5)
\notag\\=&\tau_2\tau_1(1-s_1)-\tau_2\tau_{21}s_2(1-s_1)-\tau_1\tau_{12}(1-s_{12})+\tau_1\tau_{2}s_1(1-s_{12})\notag\\
&-\tau_{12}\tau_{121}(1-s_{121})+\tau_{12}(1-\tau_{1})s_{12}(1-s_{121})-\tau_{121}\tau_{212}(1-s_{212})\notag\\
&+\tau_{121}(1-\tau_{12})s_{121}(1-s_{212})-\tau_{212}\tau_{21}(1-s_{21})+\tau_{212}(1-\tau_{121})s_{212}(1-s_{21})\notag\\
&-\tau_{21}\tau_{2}(1-s_{2})+\tau_{21}(1-\tau_{212})s_{21}(1-s_{2})+\tau_{12}(1-s_{12})\notag\\
&+\tau_{121}(1-s_{121})+\tau_{212}(1-s_{212})+\tau_{21}(1-s_{21}).
\label{g2-lemma-equatlity1}
\end{align}
It is easy to check that $\tau_2\tau_1-\tau_1\tau_{12}-\tau_{12}\tau_{121}-\tau_{121}\tau_{212}-\tau_{212}\tau_{21}-\tau_{21}\tau_{2}+\tau_{12}+\tau_{121}+\tau_{212}+\tau_{21}=0.$ By this and \eqref{g2-lemma-equatlity1}, we can see that $(1)$ holds.

$(2)$ We have
\begin{align*}
\hat{p}_{W_2}(\widetilde{D}_3\widetilde{D}_6)=\tau_{121}(1-s_{121})\tau_{2}(1-s_{2})=\tau_{121}\tau_{2}(1-s_{2})-\tau_{121}\tau_{2}s_{121}(1-s_{2});
\end{align*}
and
\begin{align*}
\hat{p}_{W_2}(\widetilde{D}_6\widetilde{D}_3)=\tau_{2}(1-s_{2})\tau_{121}(1-s_{121})=\tau_{2}\tau_{121}(1-s_{121})-\tau_{2}\tau_{121}s_{2}(1-s_{121}).
\end{align*}
By this, it is easy to check that $(2)$ holds.

$(3)$ We have
\begin{align*}
\hat{p}_{W_2}(\widetilde{D}_3\widetilde{D}_5)=\tau_{121}(1-s_{121})\tau_{21}(1-s_{21})=\tau_{121}\tau_{21}(1-s_{21})-\tau_{121}(1-\tau_{1})s_{121}(1-s_{21});
\end{align*}
\begin{align*}
\hat{p}_{W_2}(\widetilde{D}_1\widetilde{D}_3)=\tau_{1}(1-s_{1})\tau_{121}(1-s_{121})=\tau_{1}\tau_{121}(1-s_{121})-\tau_{1}\tau_{21}s_{1}(1-s_{121});
\end{align*}
\begin{align*}
\hat{p}_{W_2}(\widetilde{D}_5\widetilde{D}_1)=\tau_{21}(1-s_{21})\tau_{1}(1-s_{1})=\tau_{21}\tau_{1}(1-s_{1})-\tau_{21}\tau_{121}s_{21}(1-s_{1}).
\end{align*}
Then we have
\begin{align}
&\hat{p}_{W_2}(\widetilde{D}_1\widetilde{D}_3+\widetilde{D}_3\widetilde{D}_5-\widetilde{D}_5\widetilde{D}_1-\widetilde{D}_3)
\notag\\=&\tau_{1}\tau_{121}(1-s_{121})-\tau_{1}\tau_{21}s_{1}(1-s_{121})+\tau_{121}\tau_{21}(1-s_{21})-\tau_{121}(1-\tau_{1})s_{121}(1-s_{21})\notag\\
&-\tau_{21}\tau_{1}(1-s_{1})+\tau_{21}\tau_{121}s_{21}(1-s_{1})-\tau_{121}(1-s_{121}).
\label{g2-lemma-equatlity2}
\end{align}
By this it is easy to check that $(3)$ holds.

$(4)$ It can be proved similarly.

$(5)$-$(8)$ They can be proved similarly by a direct calculation. We skip the details.
\end{proof}

Lemma \ref{G2-lemma1} immediately yields the following corollary.
\begin{corollary}\label{G2-corollary-module}
The assignments $D_{i} \mapsto \frac{1}{1-t_{i}}(1-s_{i})$ and $s_{i}\mapsto s_{i}$ $(1\leq i\leq2),$ define a homomorphism of algebras $p_{W_2} :\mathbf{H}(W_2)\rightarrow \mathcal{Q}_2\rtimes \mathbb{Z}W_{2}.$
\end{corollary}

We have a natural action of $\mathcal{Q}_2\rtimes \mathbb{Z}W_{2}$ on $\mathcal{Q}_2$ by $(tw)(t')=t\cdot w(t')$ for $t, t'\in \mathcal{Q}_2$ and $w\in W_{2}.$ This, composing with $\hat{p}_{W_2}$, gives an action of $\hat{\bf{H}}(W_{2})$ on $\mathcal{Q}_2$ such that $\mathcal{L}_2$ is invariant and $\mathcal{Q}_2$ and $\mathcal{L}_2$ are both module algebras over $\hat{\bf{H}}(W_{2}).$ By Corollary \ref{G2-corollary-module} we get that $p_{W_2}$ defines a structure of a module algebra of $\mathbf{H}(W_2)$ on $\mathcal{Q}_2$ such that $\mathcal{L}_2$ is a module subalgebra. Thus, we verify Conjecture \ref{introduction-conjecture1} for the Coxeter group $W_{2}.$

From the presentation in Theorem \ref{theoremg2}, we can see that the elements \eqref{G2-K12W2-element1}, \eqref{G2-K12W2-element2} and \eqref{G2-K12W2-element3} do not belong to the elements $(a)$ and $(b)$ listed in Conjecture \ref{introduction-conjecture2}. Thus, Conjecture \ref{introduction-conjecture2} does not hold for the Coxeter group $W_{2}.$

\section{Type $I_{2}(5)$}
Let $W_3$ be the dihedral group of type $I_{2}(5)$ with generators $s_1,s_2$ and relations $s_1^{2}=s_2^{2}=1$ and $s_1s_2s_1s_2s_1=s_2s_1s_2s_1s_2.$ We then have the following theorem.

\begin{theorem}\label{theoremi25}
The Hecke-Hopf algebra $\mathbf{H}(W_3),$ associated to $W_3,$ has the following presentation$:$ it is generated by the elements $s_1,s_2,D_1,D_2$ with the following relations$:$
\begin{align}
s_1^{2}=&s_2^{2}=1,\quad s_1s_2s_1s_2s_1=s_2s_1s_2s_1s_2;\label{I25-relation1}\\[0.1em]
D_{i}^{2}=&D_{i};\label{I25-relation2}\\[0.1em]
s_{i}D_i&+D_{i}s_i=s_i-1;\label{I25-relation3}\\[0.1em]
D_2s_1&s_2s_1s_2=s_1s_2s_1s_2D_1;\label{I25-relation4}\\[0.1em]
D_2s_1&s_2s_1D_2=D_1D_2s_1s_2s_1+s_1D_2D_1s_2s_1+s_1D_2s_1s_2s_1\notag\\
&\qquad\qquad+s_1s_2D_1D_2s_1+s_1s_2D_1s_2s_1+s_1s_2s_1D_2D_1+s_1s_2s_1D_2s_1;\label{I25-relation5}\\[0.1em]
D_1s_2&s_1s_2D_1=D_2D_1s_2s_1s_2+s_2D_1D_2s_1s_2+s_2D_1s_2s_1s_2\notag\\
&\qquad\qquad+s_2s_1D_2D_1s_2+s_2s_1D_2s_1s_2+s_2s_1s_2D_1D_2+s_2s_1s_2D_1s_2;\label{I25-relation6}\\[0.1em]
D_1s_2&s_1D_2s_1+s_1D_2s_1s_2D_1=\notag\\
&\qquad\qquad s_2D_1s_2D_1s_2+D_2s_1D_2s_1s_2+s_2s_1D_2s_1D_2+s_2s_1D_2s_1s_2;\label{I25-relation7}\\[0.1em]
D_2D_1&D_2s_1D_2+D_2s_1D_2D_1D_2+D_2s_1D_2s_1D_2=\notag\\
&\qquad D_1D_2D_1D_2s_1+D_1D_2s_1D_2D_1+D_1D_2s_1D_2s_1+s_1D_2D_1D_2D_1\notag\\
&\qquad\qquad+s_1D_2D_1D_2s_1+s_1D_2s_1D_2D_1+s_1D_2s_1D_2s_1;\label{I25-relation8}\\[0.1em]
D_1D_2&D_1s_2D_1+D_1s_2D_1D_2D_1+D_1s_2D_1s_2D_1=\notag\\
&\qquad D_2D_1D_2D_1s_2+D_2D_1s_2D_1D_2+D_2D_1s_2D_1s_2+s_2D_1D_2D_1D_2\notag\\
&\qquad\qquad+s_2D_1D_2D_1s_2+s_2D_1s_2D_1D_2+s_2D_1s_2D_1s_2.\label{I25-relation9}
\end{align}
\end{theorem}
\begin{proof}
The proof of this theorem is exactly the same as that of Theorem \ref{theoremb1}, but much more complicated. We skip the details.
\end{proof}

Set $\overline{D}_{1} :=D_1,$ $\overline{D}_{2} :=s_1D_2s_1,$ $\overline{D}_{3} :=s_1s_2D_1s_2s_1=s_2s_1D_2s_1s_2,$ $\overline{D}_{4} :=s_2D_1s_2,$ and $\overline{D}_{5} :=D_2.$

We then have that \eqref{I25-relation5} corresponds to the element
\begin{align}\label{i25-K12W3-element1-1}
\overline{D}_{5}\overline{D}_{1}-\overline{D}_{1}\overline{D}_{2}-\overline{D}_{2}\overline{D}_{3}-\overline{D}_{3}\overline{D}_{4}-
\overline{D}_{4}\overline{D}_{5}+\overline{D}_{2}+\overline{D}_{3}+\overline{D}_{4}
\end{align}
in $K_{12}(W_3);$

\eqref{I25-relation6} corresponds to the element
\begin{align}\label{i25-K12W3-element2-2}
\overline{D}_{1}\overline{D}_{5}-\overline{D}_{2}\overline{D}_{1}-\overline{D}_{3}\overline{D}_{2}-\overline{D}_{4}\overline{D}_{3}-
\overline{D}_{5}\overline{D}_{4}+\overline{D}_{2}+\overline{D}_{3}+\overline{D}_{4}
\end{align}
in $K_{12}(W_3);$

\eqref{I25-relation7} corresponds to the element
\begin{align}\label{i25-K12W3-element3-3}
\overline{D}_{4}\overline{D}_{1}+\overline{D}_{5}\overline{D}_{2}-\overline{D}_{1}\overline{D}_{3}-\overline{D}_{2}\overline{D}_{4}-
\overline{D}_{3}\overline{D}_{5}+\overline{D}_{3}
\end{align}
in $K_{12}(W_3);$

\eqref{I25-relation8} corresponds to the element
\begin{align}\label{i25-K12W3-element4-4}
&\overline{D}_{5}\overline{D}_{1}\overline{D}_{5}\overline{D}_{2}-\overline{D}_{5}\overline{D}_{2}\overline{D}_{1}\overline{D}_{2}
+\overline{D}_{5}\overline{D}_{2}-\overline{D}_{1}\overline{D}_{5}\overline{D}_{1}\overline{D}_{5}+
\overline{D}_{1}\overline{D}_{5}\overline{D}_{2}\overline{D}_{1}\notag\\
&\qquad -\overline{D}_{1}\overline{D}_{5}\overline{D}_{2}-
\overline{D}_{2}\overline{D}_{1}\overline{D}_{2}\overline{D}_{1}+
\overline{D}_{2}\overline{D}_{1}+\overline{D}_{2}\overline{D}_{1}\overline{D}_{2}-\overline{D}_{2}
\end{align}
in $K_{12}(W_3);$

\eqref{I25-relation9} corresponds to the element
\begin{align}\label{i25-K12W3-element5-5}
&\overline{D}_{1}\overline{D}_{5}\overline{D}_{1}\overline{D}_{4}-\overline{D}_{1}\overline{D}_{4}\overline{D}_{5}\overline{D}_{4}
+\overline{D}_{1}\overline{D}_{4}-\overline{D}_{5}\overline{D}_{1}\overline{D}_{5}\overline{D}_{1}+
\overline{D}_{5}\overline{D}_{1}\overline{D}_{4}\overline{D}_{5}\notag\\
&\qquad -\overline{D}_{5}\overline{D}_{1}\overline{D}_{4}-
\overline{D}_{4}\overline{D}_{5}\overline{D}_{4}\overline{D}_{5}+
\overline{D}_{4}\overline{D}_{5}+\overline{D}_{4}\overline{D}_{5}\overline{D}_{4}-\overline{D}_{4}
\end{align}
in $K_{12}(W_3);$

From the presentation in Theorem \ref{theoremi25}, we can get the following relations in $\mathbf{H}(W_3)$.
\begin{lemma} We have, in $\mathbf{H}(W_3),$
\begin{align}
D_2s_1&s_2D_1s_2+s_2D_1s_2s_1D_2=\notag\\
&\qquad\qquad s_1D_2s_1D_2s_1+D_1s_2D_1s_2s_1+s_1s_2D_1s_2D_1+s_1s_2D_1s_2s_1;\label{I25-relation10}
\end{align}

\begin{align}
D_1D_2&s_1s_2D_1+s_1D_2D_1s_2D_1+s_1D_2s_1s_2D_1+s_1s_2D_1D_2D_1+s_1s_2D_1s_2D_1=\notag\\
&\qquad D_2D_1D_2s_1s_2+D_2s_1D_2D_1s_2+D_2s_1D_2s_1s_2+D_2s_1s_2D_1D_2+D_2s_1s_2D_1s_2;\label{I25-relation11}
\end{align}

\begin{align}
D_2D_1&s_2s_1D_2+s_2D_1D_2s_1D_2+s_2D_1s_2s_1D_2+s_2s_1D_2D_1D_2+s_2s_1D_2s_1D_2=\notag\\
&\qquad D_1D_2D_1s_2s_1+D_1s_2D_1D_2s_1+D_1s_2D_1s_2s_1+D_1s_2s_1D_2D_1+D_1s_2s_1D_2s_1;\label{I25-relation12}
\end{align}

\begin{align}
s_1s_2&s_1D_2D_1+s_1s_2D_1s_2D_1+s_1s_2D_1D_2s_1+s_1D_2s_1s_2D_1+s_1D_2s_1D_2s_1+\notag\\
&s_1D_2D_1s_2s_1+D_1s_2s_1s_2D_1+D_1s_2s_1D_2s_1+D_1s_2D_1s_2s_1+D_1D_2s_1s_2s_1=\notag\\
s_2s_1&s_2D_1D_2+s_2s_1D_2s_1D_2+s_2s_1D_2D_1s_2+s_2D_1s_2s_1D_2+s_2D_1s_2D_1s_2+\notag\\
&s_2D_1D_2s_1s_2+D_2s_1s_2s_1D_2+D_2s_1s_2D_1s_2+D_2s_1D_2s_1s_2+D_2D_1s_2s_1s_2;\label{I25-relation13}
\end{align}

\begin{align}
s_2D_1&s_2D_1D_2+s_2D_1D_2D_1s_2+D_2s_1D_2s_1D_2+D_2D_1s_2D_1s_2+s_2D_1s_2D_1s_2=\notag\\
&\qquad s_1D_2s_1D_2D_1+s_1D_2D_1D_2s_1+D_1s_2D_1s_2D_1+D_1D_2s_1D_2s_1+s_1D_2s_1D_2s_1;\label{I25-relation14}
\end{align}

\begin{align}
s_1s_2&D_1D_2D_1+s_1D_2s_1D_2D_1+s_1D_2D_1s_2D_1+s_1D_2D_1D_2s_1+D_1s_2s_1D_2D_1+\notag\\
&D_1s_2D_1s_2D_1+D_1s_2D_1D_2s_1+D_1D_2s_1s_2D_1+D_1D_2s_1D_2s_1+D_1D_2D_1s_2s_1=\notag\\
s_2s_1&D_2D_1D_2+s_2D_1s_2D_1D_2+s_2D_1D_2s_1D_2+s_2D_1D_2D_1s_2+D_2s_1s_2D_1D_2+\notag\\
&D_2s_1D_2s_1D_2+D_2s_1D_2D_1s_2+D_2D_1s_2s_1D_2+D_2D_1s_2D_1s_2+D_2D_1D_2s_1s_2;\label{I25-relation15}
\end{align}

\begin{align}
&s_1D_2D_1D_2D_1+D_1s_2D_1D_2D_1+D_1D_2s_1D_2D_1+D_1D_2D_1s_2D_1+D_1D_2D_1D_2s_1=\notag\\
&\qquad s_2D_1D_2D_1D_2+D_2s_1D_2D_1D_2+D_2D_1s_2D_1D_2+D_2D_1D_2s_1D_2+D_2D_1D_2D_1s_2;\label{I25-relation16}
\end{align}

\begin{align}
D_1D_2D_1D_2D_1=D_2D_1D_2D_1D_2.\label{I25-relation17}
\end{align}
\end{lemma}
\begin{proof}
\eqref{I25-relation10} By \eqref{I25-relation7} we get that
\begin{align}
s_1D_1s_2&s_1D_2s_1s_2+D_2s_1s_2D_1s_2=\notag\\
&\qquad\qquad s_1s_2D_1s_2D_1+s_1D_2s_1D_2s_1+s_1s_2s_1D_2s_1D_2s_2+s_1s_2s_1D_2s_1,\label{i25relation1}
\end{align}
from which we can easily get \eqref{I25-relation10}.

\eqref{I25-relation11} By \eqref{I25-relation5} we get that
\begin{align}
D_2D_1&D_2s_1s_2+D_2s_1D_2D_1s_2+D_2s_1D_2s_1s_2+D_2s_1s_2D_1D_2\notag\\
&\qquad\qquad+D_2s_1s_2D_1s_2+D_2s_1s_2s_1D_2D_1s_1+D_2s_1s_2s_1D_2=\notag\\
&D_1D_2s_1s_2+s_1D_2D_1s_2+s_1D_2s_1s_2+s_1s_2D_1D_2\notag\\
&\qquad\qquad+s_1s_2D_1s_2+s_1s_2s_1D_2D_1s_1+s_1s_2s_1D_2.\label{i25relation2}
\end{align}
Note that $D_2s_1s_2s_1D_2D_1s_1+D_2s_1s_2s_1D_2=-D_2s_1s_2s_1D_2s_1D_1+D_2s_1s_2s_1D_2s_1,$ from which we can easily get \eqref{I25-relation11}.

\eqref{I25-relation12} It can be proved similarly.

\eqref{I25-relation13} By \eqref{I25-relation7} and \eqref{I25-relation10} we get that \eqref{I25-relation13} is equivalent to the following equality:
\begin{align}
s_1s_2&s_1D_2D_1+s_1s_2D_1D_2s_1+s_1D_2D_1s_2s_1+D_1s_2s_1s_2D_1+D_1D_2s_1s_2s_1=\notag\\
&s_2s_1s_2D_1D_2+s_2s_1D_2D_1s_2+s_2D_1D_2s_1s_2+D_2s_1s_2s_1D_2+D_2D_1s_2s_1s_2.\label{i25relation3}
\end{align}
By \eqref{I25-relation5} and \eqref{I25-relation6} we get that \eqref{i25relation3} is equivalent to the following equality:
\begin{align}
s_2D_1s_2s_1s_2+s_2s_1D_2s_1s_2+s_2s_1s_2D_1s_2=s_1D_2s_1s_2s_1+s_1s_2D_1s_2s_1+s_1s_2s_1D_2s_1,\label{i25relation4}
\end{align}
which can be easily deduced from \eqref{I25-relation4}.

\eqref{I25-relation14} It is easily to see that \eqref{I25-relation14} is equivalent to the following equality:
\begin{align}
\overline{D}_{4}\overline{D}_{1}\overline{D}_{5}&-\overline{D}_{4}\overline{D}_{5}\overline{D}_{4}+
\overline{D}_{4}+\overline{D}_{5}\overline{D}_{2}\overline{D}_{5}+\overline{D}_{5}\overline{D}_{1}\overline{D}_{4}=\notag\\
&\overline{D}_{2}\overline{D}_{5}\overline{D}_{1}-\overline{D}_{2}\overline{D}_{1}\overline{D}_{2}+\overline{D}_{2}+
\overline{D}_{1}\overline{D}_{4}\overline{D}_{1}+\overline{D}_{1}\overline{D}_{5}\overline{D}_{2}.\label{i25relation5}
\end{align}
By \eqref{I25-relation7} we have that
\begin{align}
\overline{D}_{4}\overline{D}_{1}+\overline{D}_{5}\overline{D}_{2}-\overline{D}_{1}\overline{D}_{3}-\overline{D}_{2}\overline{D}_{4}-
\overline{D}_{3}\overline{D}_{5}+\overline{D}_{3}=0.\label{i25relation6}
\end{align}
By \eqref{i25relation6} we have that
\begin{align}
\overline{D}_{1}\overline{D}_{4}\overline{D}_{1}+\overline{D}_{1}\overline{D}_{5}\overline{D}_{2}-\overline{D}_{1}\overline{D}_{2}\overline{D}_{4}-
\overline{D}_{1}\overline{D}_{3}\overline{D}_{5}=0,\label{i25relation7}
\end{align}
and
\begin{align}
\overline{D}_{4}\overline{D}_{1}\overline{D}_{5}+\overline{D}_{5}\overline{D}_{2}\overline{D}_{5}-
\overline{D}_{1}\overline{D}_{3}\overline{D}_{5}-\overline{D}_{2}\overline{D}_{4}\overline{D}_{5}=0.\label{i25relation8}
\end{align}
Combing \eqref{i25relation7} and \eqref{i25relation8} we have that
\begin{align}
\overline{D}_{1}\overline{D}_{4}\overline{D}_{1}+\overline{D}_{1}\overline{D}_{5}\overline{D}_{2}+\overline{D}_{2}\overline{D}_{4}\overline{D}_{5}=
\overline{D}_{4}\overline{D}_{1}\overline{D}_{5}+\overline{D}_{5}\overline{D}_{2}\overline{D}_{5}+
\overline{D}_{1}\overline{D}_{2}\overline{D}_{4}.\label{i25relation9}
\end{align}
By \eqref{I25-relation5} we get that
\begin{align}
\overline{D}_{5}\overline{D}_{1}-\overline{D}_{1}\overline{D}_{2}-\overline{D}_{2}\overline{D}_{3}-\overline{D}_{3}\overline{D}_{4}-
\overline{D}_{4}\overline{D}_{5}+\overline{D}_{2}+\overline{D}_{3}+\overline{D}_{4}=0.\label{i25relation10}
\end{align}
By \eqref{i25relation10} we get that
\begin{align}
\overline{D}_{5}\overline{D}_{1}\overline{D}_{4}-\overline{D}_{1}\overline{D}_{2}\overline{D}_{4}-
\overline{D}_{2}\overline{D}_{3}\overline{D}_{4}-\overline{D}_{4}\overline{D}_{5}\overline{D}_{4}
+\overline{D}_{2}\overline{D}_{4}+\overline{D}_{4}=0.\label{i25relation11}
\end{align}
and
\begin{align}
\overline{D}_{2}\overline{D}_{5}\overline{D}_{1}-\overline{D}_{2}\overline{D}_{1}\overline{D}_{2}-\overline{D}_{2}\overline{D}_{3}\overline{D}_{4}-
\overline{D}_{2}\overline{D}_{4}\overline{D}_{5}+\overline{D}_{2}+\overline{D}_{2}\overline{D}_{4}=0.\label{i25relation12}
\end{align}
Combing \eqref{i25relation11} and \eqref{i25relation12} we get that
\begin{align}
\overline{D}_{2}\overline{D}_{5}\overline{D}_{1}-\overline{D}_{2}\overline{D}_{1}\overline{D}_{2}-
\overline{D}_{2}\overline{D}_{4}\overline{D}_{5}+\overline{D}_{2}=\overline{D}_{5}\overline{D}_{1}\overline{D}_{4}
-\overline{D}_{1}\overline{D}_{2}\overline{D}_{4}
-\overline{D}_{4}\overline{D}_{5}\overline{D}_{4}
+\overline{D}_{4}.\label{i25relation13}
\end{align}
By combing \eqref{i25relation9} and \eqref{i25relation13} we can easily see that \eqref{i25relation5} holds.

\eqref{I25-relation15} By \eqref{I25-relation11} and \eqref{I25-relation12} we have that \eqref{I25-relation15} is equivalent to the following equality:
\begin{align}
&s_2D_1s_2D_1D_2+s_2D_1D_2D_1s_2+D_2s_1D_2s_1D_2+D_2D_1s_2D_1s_2\notag\\
&\qquad -D_2s_1D_2s_1s_2-D_2s_1s_2D_1s_2-s_2D_1s_2s_1D_2-s_2s_1D_2s_1D_2=\notag\\
&s_1D_2s_1D_2D_1+s_1D_2D_1D_2s_1+D_1s_2D_1s_2D_1+D_1D_2s_1D_2s_1\notag\\
&\qquad -s_1D_2s_1s_2D_1-s_1s_2D_1s_2D_1-D_1s_2D_1s_2s_1-D_1s_2s_1D_2s_1.\label{i25relation14}
\end{align}
By \eqref{I25-relation7} and \eqref{I25-relation10} it is easy to see that \eqref{i25relation14} is equivalent to \eqref{I25-relation14}.

\eqref{I25-relation16} By \eqref{I25-relation8} and \eqref{I25-relation9} we have that
\begin{align}
&D_1D_2D_1D_2s_1+D_1D_2s_1D_2D_1+D_1D_2s_1D_2s_1+s_1D_2D_1D_2D_1+s_1D_2D_1D_2s_1\notag\\
&\qquad +s_1D_2s_1D_2D_1+s_1D_2s_1D_2s_1+D_1D_2D_1s_2D_1+D_1s_2D_1D_2D_1+D_1s_2D_1s_2D_1=\notag\\
&D_2D_1D_2s_1D_2+D_2s_1D_2D_1D_2+D_2s_1D_2s_1D_2+D_2D_1D_2D_1s_2+D_2D_1s_2D_1D_2\notag\\
&\qquad +D_2D_1s_2D_1s_2+s_2D_1D_2D_1D_2+s_2D_1D_2D_1s_2+s_2D_1s_2D_1D_2+s_2D_1s_2D_1s_2.\label{i25relation15}
\end{align}
Combing \eqref{I25-relation14} and \eqref{i25relation15} it is easy to see that \eqref{I25-relation16} holds.

\eqref{I25-relation17} By \eqref{I25-relation8} we have that
\begin{align}
D_2D_1&D_2D_1D_2+D_2D_1D_2s_1D_2D_1s_1+D_2D_1D_2s_1D_2+D_2s_1D_2D_1D_2D_1s_1\notag\\
&\qquad+D_2s_1D_2D_1D_2+D_2s_1D_2s_1D_2D_1s_1+D_2s_1D_2s_1D_2=\notag\\
&\qquad\qquad D_2D_1D_2s_1D_2s_1+D_2s_1D_2D_1D_2s_1+D_2s_1D_2s_1D_2s_1.\label{i25relation16}
\end{align}
By \eqref{I25-relation8} and \eqref{i25relation16} we have that
\begin{align}
D_2D_1D_2&D_1D_2=(D_2D_1D_2s_1D_2+D_2s_1D_2D_1D_2+D_2s_1D_2s_1D_2)s_1D_1=\notag\\
&D_1D_2D_1D_2D_1+D_1D_2s_1D_2D_1s_1D_1+D_1D_2s_1D_2D_1+s_1D_2D_1D_2D_1s_1D_1\notag\\
&\qquad +s_1D_2D_1D_2D_1+s_1D_2s_1D_2D_1s_1D_1+s_1D_2s_1D_2D_1\notag\\
&\qquad \qquad =D_1D_2D_1D_2D_1.\label{i25relation17}
\end{align}
\end{proof}

\begin{remark}
Recall from [BK, (1.3)] that it has been shown that the following relation holds in ${\bf D}(W_{3})$:
\begin{align}\label{i25-remark-relati123}
\overline{D}_{1}\overline{D}_{2}\overline{D}_{3}\overline{D}_{4}+(\overline{D}_{1}\overline{D}_{2}\overline{D}_{4}+
&\overline{D}_{2}\overline{D}_{3}\overline{D}_{4}-\overline{D}_{2}\overline{D}_{4})(\overline{D}_{5}-1)\notag\\
&\qquad=\overline{D}_{5}\overline{D}_{4}\overline{D}_{3}\overline{D}_{1}+
\overline{D}_{5}\overline{D}_{3}\overline{D}_{2}\overline{D}_{1}-\overline{D}_{5}\overline{D}_{3}\overline{D}_{1}.
\end{align}

If we admit that the equality \eqref{I25-relation6} in Theorem \ref{theoremi25} holds, that is, the element \eqref{i25-K12W3-element2-2} equals zero, then by a direct calculation it is easy to see that the right-hand side of \eqref{i25-remark-relati123} equals $\overline{D}_{5}\overline{D}_{1}\overline{D}_{5}\overline{D}_{1},$ and \eqref{i25-remark-relati123} is equivalent to the equality \eqref{I25-relation8}, that is, the fact that the element \eqref{i25-K12W3-element4-4} equals zero. It is easy to see that the element \eqref{i25-K12W3-element4-4} does not belong to the elements $(a)$ and $(b)$ listed in Conjecture \ref{introduction-conjecture2}. Thus, Conjecture \ref{introduction-conjecture2} does not hold for the Coxeter group $W_{3}.$
\end{remark}

\section{Proof of Conjecture \ref{introduction-conjecture1} and counterexample of Conjecture \ref{introduction-conjecture2}}
In this section we first give the proof of Conjecture \ref{introduction-conjecture1} by using the results in Sections 2 and 3.

Combining [BK, Theorem 1.24] with Theorems \ref{theoremb1}, \ref{theoremg2} and \ref{theoremi25}, we immediately get the following result.
\begin{theorem}\label{representation-theorem}
Assume that $W$ is a Coxeter group generated by the elements $s_i$ $(i\in I)$ such that the parameter $m_{ij},$ associated to any two distinct vertices $i$ and $j$ in the presentation of $W,$ is less than or equal to $6.$

Associated to $W,$ the Hecke-Hopf algebra $\mathbf{H}(W)$ has the following presentation$:$ it is generated by the elements $s_i$ and $D_{i}$ $(i\in I)$ with the following relations$:$
\begin{align}\label{representation-relation1-1}
s_i^2=1,\quad s_iD_i+D_is_i=s_i-1,\quad D_i^2=D_i \text{ for } i\in I;
\end{align}
if $m_{ij}=2$, we have
\begin{align}\label{representation-relation2-2}
s_js_i=s_is_j,\quad D_js_i=s_iD_j,\quad D_jD_i=D_iD_j;
\end{align}
if $m_{ij}=3$, we have
\begin{align}\label{representation-relation3-3}
s_js_is_j=s_is_js_i,\quad D_is_js_i=s_js_iD_j,\quad D_js_iD_j=s_iD_jD_i+D_iD_js_i+s_iD_js_i;
\end{align}
if $m_{ij}=4$, we have
\begin{align}\label{representation-relation4-4-1}
s_is_js_is_j=s_js_is_js_i,\quad D_is_js_is_j=s_js_is_jD_i,
\end{align}
\begin{align}\label{representation-relation4-4-2}
D_{j}s_is_{j}D_i=D_{i}D_{j}s_is_j+s_{i}D_{j}D_is_j+s_is_{j}D_{i}D_j+s_is_{j}D_is_j+s_{i}D_{j}s_is_j,
\end{align}
\begin{align}\label{representation-relation4-4-3}
D_{i}D_{j}D_{i}D_{j}=D_{j}D_{i}D_{j}D_{i};
\end{align}
if $m_{ij}=5$, we have
\begin{align}\label{representation-relation6-6-1}
s_{i}s_{j}s_{i}s_{j}s_{i}=s_{j}s_{i}s_{j}s_{i}s_{j},\quad D_js_{i}s_{j}s_{i}s_{j}=s_{i}s_{j}s_{i}s_{j}D_i,
\end{align}
\begin{align}\label{representation-relation6-6-2}
D_is_{j}&s_{i}s_{j}D_i=D_jD_is_{j}s_{i}s_{j}+s_{j}D_iD_js_{i}s_{j}+s_{j}D_is_{j}s_{i}s_{j}\notag\\
&\qquad\qquad+s_{j}s_{i}D_jD_is_{j}+s_{j}s_{i}D_js_{i}s_{j}+s_{j}s_{i}s_{j}D_iD_j+s_{j}s_{i}s_{j}D_is_{j},
\end{align}
\begin{align}\label{representation-relation6-6-3}
D_is_{j}&s_{i}D_js_{i}+s_{i}D_js_{i}s_{j}D_i=\notag\\
&\qquad\qquad s_{j}D_is_{j}D_is_{j}+D_js_{i}D_js_{i}s_{j}+s_{j}s_{i}D_js_{i}D_j+s_{j}s_{i}D_js_{i}s_{j},
\end{align}
\begin{align}\label{representation-relation6-6-4}
D_iD_j&D_is_{j}D_i+D_is_{j}D_iD_jD_i+D_is_{j}D_is_{j}D_i=\notag\\
&\qquad D_jD_iD_jD_is_{j}+D_jD_is_{j}D_iD_j+D_jD_is_{j}D_is_{j}+s_{j}D_iD_jD_iD_j\notag\\
&\qquad\qquad+s_{j}D_iD_jD_is_{j}+s_{j}D_is_{j}D_iD_j+s_{j}D_is_{j}D_is_{j};
\end{align}
if $m_{ij}=6$, we have
\begin{align}\label{representation-relation5-5-1}
s_is_js_is_js_is_j=s_js_is_js_is_js_i,\quad D_is_js_is_js_is_j=s_js_is_js_is_jD_i,
\end{align}
\begin{align}\label{representation-relation5-5-2}
D_is_js_i&s_js_iD_j=D_jD_is_js_is_js_i+s_jD_iD_js_is_js_i+s_js_iD_jD_is_js_i\notag\\
&\qquad \qquad +s_js_is_jD_iD_js_i+s_js_is_js_iD_jD_i+s_js_iD_js_is_js_i\notag\\
&\qquad \qquad +s_js_is_js_iD_js_i+s_jD_is_js_is_js_i+s_js_is_jD_is_js_i,
\end{align}
\begin{align}\label{representation-relation5-5-3}
D_js_{i}s_{j}D_is_{j}s_{i}=s_{i}s_{j}D_is_{j}s_{i}D_j,
\end{align}
\begin{align}\label{representation-relation5-5-4}
D_is_{j}s_{i}s_{j}D_is_{j}=s_{j}s_{i}s_{j}D_is_{j}D_i+s_{j}D_is_{j}D_is_{j}s_{i}+s_{j}s_{i}s_{j}D_is_{j}s_{i},
\end{align}
\begin{align}\label{representation-relation5-5-5}
D_jD_i&D_js_{i}s_{j}D_i+D_js_{i}D_jD_is_{j}D_i+D_js_{i}D_js_{i}s_{j}D_i\notag\\
&\qquad\qquad\qquad +D_js_{i}s_{j}D_iD_jD_i+D_js_{i}s_{j}D_is_{j}D_i=\notag\\
&D_iD_jD_iD_js_{i}s_{j}+D_iD_js_{i}D_jD_is_{j}+D_iD_js_{i}D_js_{i}s_{j}+D_iD_js_{i}s_{j}D_iD_j\notag\\
&+D_iD_js_{i}s_{j}D_is_{j}+s_{i}D_jD_iD_jD_is_{j}+s_{i}D_jD_iD_js_{i}s_{j}+s_{i}D_jD_is_{j}D_iD_j\notag\\
&+s_{i}D_jD_is_{j}D_is_{j}+s_{i}D_js_{i}D_jD_is_{j}+s_{i}D_js_{i}D_js_{i}s_{j}+s_{i}D_js_{i}s_{j}D_iD_j\notag\\
&+s_{i}D_js_{i}s_{j}D_is_{j}+s_{i}s_{j}D_iD_jD_iD_j+s_{i}s_{j}D_iD_jD_is_{j}+s_{i}s_{j}D_is_{j}D_iD_j\notag\\
&+s_{i}s_{j}D_is_{j}D_is_{j},
\end{align}
\begin{align}\label{representation-relation5-5-6}
s_{i}D_js_{i}&D_jD_iD_j+s_{i}D_jD_iD_js_{i}D_j+D_iD_js_{i}D_js_{i}D_j=\notag\\
&\qquad \qquad \qquad D_js_{i}D_js_{i}D_jD_i+D_js_{i}D_jD_iD_js_{i}+D_jD_iD_js_{i}D_js_{i},
\end{align}
\begin{align}\label{representation-relation5-5-7}
D_iD_jD_iD_jD_iD_j=D_jD_iD_jD_iD_jD_i.
\end{align}
\end{theorem}

\noindent{\bf Verification of Conjecture \ref{introduction-conjecture1}} Combining the results in Sections 2 and 3, we immediately get that Conjecture \ref{introduction-conjecture1} is true when the Coxeter group $W$ is crystallographic and non-simply-laced, that is, the parameter $m_{ij},$ associated to any two distinct vertices $i$ and $j$ in the presentation of $W$, equals $4$ and $6$.

\noindent{\bf Counterexample to Conjecture \ref{introduction-conjecture2}} By the results in Section 2, we immediately get that Conjecture \ref{introduction-conjecture2} is true when the parameter $m_{ij},$ associated to any two distinct vertices $i$ and $j$ in the presentation of the Coxeter group $W$, is less than or equal to $4.$ By the results in Section 3, we see that Conjecture \ref{introduction-conjecture2} does not hold when there exists some parameter $m_{kl},$ associated to two distinct vertices $k$ and $l$ in the presentation of $W$, which equals $6.$

Recall that in the definition of the Hecke-Hopf algebra $\mathbf{H}(W)$, it has been claimed that $\mathbf{H}(W)$ is a Hopf algebra (see [BK, Theorem 1.23]). In fact we have the following result.
\begin{theorem}\label{hopf-algebra-structure}
In the assumption of Theorem \ref{representation-theorem}, the Hecke-Hopf algebra $\mathbf{H}(W)$ is a Hopf algebra with the coproduct $\Delta,$ the counit $\varepsilon,$ and the antipode $S$ given respectively by $($for $i\in I):$
$$\Delta(s_i)=s_i\otimes s_i,~\Delta(D_i)=D_i\otimes 1+s_i\otimes D_i,~\varepsilon(w)=1, ~\varepsilon(D_i)=0,~ S(s_i)=s_i,~S(D_i)=-s_iD_i.$$
\end{theorem}

\section{Generalized affine Hecke algebras}
Let $G$ be a connected reductive group over $\mathbb{C}$ and $T$ a maximal torus of $G.$ Let $N_{G}(T)$ be the normalizer of $T$ in $G.$ Then $W_{0}=N_{G}(T)/T$ is a Weyl group, which acts on the co-character group $P^{\vee} :=\mathrm{Hom}(\mathbb{C}^{*}, T).$ Consider the semi-direct product $W=P^{\vee}\rtimes W_{0}.$ Let $R^{\vee}$ be the set of coroots, which spans the coroot lattice $Q^{\vee}.$ Assume that $\{s_{i}\:|\:i\in I\}$ is a set of simple reflections in $W_{0}$ and $\{\alpha_{i}^{\vee}\:|\:i\in I\}$ is a set of simple coroots in $R^{\vee}.$ For the rest of this section we will assume that $G$ is semisimple and simply connected.

We now define the affine Hecke algebra associated to $W$ following J. Bernstein (unpublished; see [Lu]).
\begin{definition}\label{definition-affine-Heck-algs}
Given a commutative unital ring $\Bbbk$ and $\underline{\bf q}=(q_{i})\in \Bbbk^{I}$ such that $q_{i}=q_{j}$ whenever $s_{i}$ and $s_{j}$ are conjugate in $W_{0}.$ The affine Hecke algebra $H_{\underline{\bf q}}(W)$ is a $\Bbbk$-algebra generated by the elements $T_{i}$ $(i\in I)$ and $X^{\lambda}$ $(\lambda\in P^{\vee})$ with the following relations:

$(1)$ $(T_{i}-1)(T_{i}+q_{i})=0$ for $i\in I;$\vspace{0.1cm}

$(2)$ $\underbrace{T_iT_jT_{i}\cdots}_{m_{ij}} =\underbrace{T_jT_iT_{j}\cdots}_{m_{ij}}$ for all distinct $i, j\in I;$\vspace{0.1cm}

$(3)$ $X^{\lambda}X^{\lambda'}=X^{\lambda'}X^{\lambda}$ for any $\lambda, \lambda'\in P^{\vee};$\vspace{0.1cm}

$(4)$ $T_{i}X^{\lambda}=X^{s_{i}(\lambda)}T_{i}+(1-q_{i})\frac{X^{\lambda}-X^{s_{i}(\lambda)}}{1-X^{-\alpha_{i}^{\vee}}}$ for $i\in I$ and $\lambda\in P^{\vee}.$
\end{definition}

Associated to $W,$ we next define the generalized affine Hecke algebra $\widetilde{{\bf H}}(W)$ over $\mathbb{Z}$ as follows.
\begin{definition}\label{definition-gene}
The generalized affine Hecke algebra $\widetilde{{\bf H}}(W)$ is a $\mathbb{Z}$-algebra generated by the elements $s_i,$ $D_{i}$ $(i\in I)$ and $X^{\lambda}$ $(\lambda\in P^{\vee})$ with the following relations$:$
\begin{align}\label{representation-relation6-1-1}
s_i^2=1,\quad s_iD_i+D_is_i=s_i-1,\quad D_i^2=D_i \text{ for } i\in I;
\end{align}
if $m_{ij}=2$, we have
\begin{align}\label{representation-relation6-2-2}
s_js_i=s_is_j,\quad D_js_i=s_iD_j,\quad D_jD_i=D_iD_j;
\end{align}
if $m_{ij}=3$, we have
\begin{align}\label{representation-relation6-3-3}
s_js_is_j=s_is_js_i,\quad D_is_js_i=s_js_iD_j,\quad D_js_iD_j=s_iD_jD_i+D_iD_js_i+s_iD_js_i;
\end{align}
if $m_{ij}=4$, we have
\begin{align}\label{representation-relation6-4-4-1}
s_is_js_is_j=s_js_is_js_i,\quad D_is_js_is_j=s_js_is_jD_i,
\end{align}
\begin{align}\label{representation-relation6-4-4-2}
D_{j}s_is_{j}D_i=D_{i}D_{j}s_is_j+s_{i}D_{j}D_is_j+s_is_{j}D_{i}D_j+s_is_{j}D_is_j+s_{i}D_{j}s_is_j,
\end{align}
\begin{align}\label{representation-relation6-4-4-3}
D_{i}D_{j}D_{i}D_{j}=D_{j}D_{i}D_{j}D_{i};
\end{align}
if $m_{ij}=6$, we have
\begin{align}\label{representation-relation6-5-5-1}
s_is_js_is_js_is_j=s_js_is_js_is_js_i,\quad D_is_js_is_js_is_j=s_js_is_js_is_jD_i,
\end{align}
\begin{align}\label{representation-relation5-5-2}
D_is_js_i&s_js_iD_j=D_jD_is_js_is_js_i+s_jD_iD_js_is_js_i+s_js_iD_jD_is_js_i\notag\\
&\qquad \qquad +s_js_is_jD_iD_js_i+s_js_is_js_iD_jD_i+s_js_iD_js_is_js_i\notag\\
&\qquad \qquad +s_js_is_js_iD_js_i+s_jD_is_js_is_js_i+s_js_is_jD_is_js_i,
\end{align}
\begin{align}\label{representation-relation6-5-5-3}
D_js_{i}s_{j}D_is_{j}s_{i}=s_{i}s_{j}D_is_{j}s_{i}D_j,
\end{align}
\begin{align}\label{representation-relation6-5-5-4}
D_is_{j}s_{i}s_{j}D_is_{j}=s_{j}s_{i}s_{j}D_is_{j}D_i+s_{j}D_is_{j}D_is_{j}s_{i}+s_{j}s_{i}s_{j}D_is_{j}s_{i},
\end{align}
\begin{align}\label{representation-relation6-5-5-5}
D_jD_i&D_js_{i}s_{j}D_i+D_js_{i}D_jD_is_{j}D_i+D_js_{i}D_js_{i}s_{j}D_i\notag\\
&\qquad\qquad\qquad +D_js_{i}s_{j}D_iD_jD_i+D_js_{i}s_{j}D_is_{j}D_i=\notag\\
&D_iD_jD_iD_js_{i}s_{j}+D_iD_js_{i}D_jD_is_{j}+D_iD_js_{i}D_js_{i}s_{j}+D_iD_js_{i}s_{j}D_iD_j\notag\\
&+D_iD_js_{i}s_{j}D_is_{j}+s_{i}D_jD_iD_jD_is_{j}+s_{i}D_jD_iD_js_{i}s_{j}+s_{i}D_jD_is_{j}D_iD_j\notag\\
&+s_{i}D_jD_is_{j}D_is_{j}+s_{i}D_js_{i}D_jD_is_{j}+s_{i}D_js_{i}D_js_{i}s_{j}+s_{i}D_js_{i}s_{j}D_iD_j\notag\\
&+s_{i}D_js_{i}s_{j}D_is_{j}+s_{i}s_{j}D_iD_jD_iD_j+s_{i}s_{j}D_iD_jD_is_{j}+s_{i}s_{j}D_is_{j}D_iD_j\notag\\
&+s_{i}s_{j}D_is_{j}D_is_{j},
\end{align}
\begin{align}\label{representation-relation6-5-5-6}
s_{i}D_js_{i}&D_jD_iD_j+s_{i}D_jD_iD_js_{i}D_j+D_iD_js_{i}D_js_{i}D_j=\notag\\
&\qquad \qquad \qquad D_js_{i}D_js_{i}D_jD_i+D_js_{i}D_jD_iD_js_{i}+D_jD_iD_js_{i}D_js_{i},
\end{align}
\begin{align}\label{representation-relation6-5-5-7}
D_iD_jD_iD_jD_iD_j=D_jD_iD_jD_iD_jD_i;
\end{align}
\begin{align}\label{representation-relation6-6-6-1}
s_iX^{\lambda}=X^{s_{i}(\lambda)}s_i\text{ for }i\in I\text{ and }\lambda\in P^{\vee};
\end{align}
\begin{align}\label{representation-relation6-7-7-1}
D_{i}X^{\lambda}=X^{s_{i}(\lambda)}D_{i}+\frac{X^{\lambda}-X^{s_{i}(\lambda)}}{1-X^{-\alpha_{i}^{\vee}}}\text{ for }i\in I\text{ and }\lambda\in P^{\vee}.
\end{align}
\end{definition}

The following theorem gives a PBW basis of $\widetilde{{\bf H}}(W).$
\begin{theorem}\label{theorem-pbw-basis}
The multiplication map defines an isomorphism of $\mathbb{Z}$-modules
\begin{align}\label{isomorphism-zmodules}
{\bf D}(W_{0})\otimes \mathbb{Z}[W_{0}]\otimes \mathbb{Z}[P^{\vee}]\overset{\sim}{\longrightarrow}\widetilde{{\bf H}}(W).
\end{align}
\end{theorem}

The following theorem is the main result of this section.
\begin{theorem}\label{theorem-injective-homomor}
The assignments
\begin{align*}
T_{i}\mapsto s_{i}+(1-q_{i})D_{i} ~~~\text{ and }~~~X^{\lambda}\mapsto X^{\lambda},
\end{align*}
for $i\in I$ and $\lambda\in P^{\vee},$ define an injective homomorphism of $\Bbbk$-algebras $\underline{\varphi}_{W} :H_{\underline{\bf q}}(W)\rightarrow \Bbbk\otimes_{\mathbb{Z}} \widetilde{{\bf H}}(W).$
\end{theorem}

The following result shows the existence of finite-dimensional $\widetilde{{\bf H}}(W)$-modules.
\begin{theorem}\label{theorem-modules-existence}
We have a natural action of $\widetilde{{\bf H}}(W)$ on $\mathbb{Z}[P^{\vee}]$ such that
\begin{align*}
(X^{\lambda}w)(X^{\lambda'})=X^{\lambda}\cdot w(X^{\lambda'})
\end{align*}
for $\lambda, \lambda'\in P^{\vee}$ and $w\in W_{0},$
and
\begin{align*}
D_{i}(f)=\frac{f-{}^{s_{i}}f}{1-X^{-\alpha_{i}^{\vee}}}
\end{align*}
for $f\in \mathbb{Z}[P^{\vee}].$
\end{theorem}

\noindent{\bf Acknowledgements.}
The author sincerely thanks Dr. Junbin Dong for his help. The author was partially supported by the National Natural Science Foundation of China (No. 11601273).



\end{document}